\documentclass{amsart}

\usepackage{mathrsfs}	

\usepackage{amssymb}	
\usepackage{amsthm}
\theoremstyle{definition}
\newtheorem{theorem}{Theorem}[section]
\newtheorem{lemma}[theorem]{Lemma}
\newtheorem{proposition}[theorem]{Proposition}
\newtheorem{corollary}[theorem]{Corollary}
\newtheorem{remark}[theorem]{Remark}
\newtheorem{definition}[theorem]{Definition}
\newtheorem{example}[theorem]{Example}
\newtheorem{convention}[theorem]{Convention}
\newtheorem{notation}[theorem]{Notation}

\usepackage{tikz-cd}

\usepackage[colorlinks=true]{hyperref}		
\hypersetup{urlcolor=blue, citecolor=blue, linkcolor=red}

\DeclareMathOperator{\mdim}{mdim}
\DeclareMathOperator{\supp}{supp}
\DeclareMathOperator{\Aut}{Aut}

\title{Large Subalgebras of Crossed Products of $C(X)$-Algebras}

\author{Xiaochun Fang}
\address{School of Mathematical sciences, Tongji University,
	Shanghai, 200092, China}
\curraddr{}
\email{xfang@tongji.edu.cn}
\thanks{}

\author{N.~C.\  Phillips}
\address{Department of Mathematics, University of Oregon,
	Eugene OR 97403-1222, USA}
\email{ncp@darkwing.uoregon.edu}
\thanks{}

\author{Junqi Yang*}
\address{School of Mathematical sciences, Tongji University,
	Shanghai, 200092, China}
\curraddr{}
\email{yangjq24@tongji.edu.cn}
\thanks{}

\date{}


\begin{document}

\maketitle

\begin{abstract}
In this paper we construct large subalgebras of
crossed product C*-algebras
of noncommutative C*-dynamics from ideals.
We apply our results to study locally trivial
unital $C(X)$-algebras such as mapping tori.
\end{abstract}

\section{Introduction}

Large subalgebras were introduced in \cite{Phillips2014large}
by the second author.
They are an abstraction of Putnam's orbit breaking subalgebra of
the crossed product of the Cantor set by a minimal homeomorphism
in \cite{Putnam1989Cast}.
Putnam's construction plays a key role in work on the structure of
the crossed product C*-algebra.
Many results were obtained directly or implicitly by using
the spirit of Putnam's construction.
Here are some of them.
Let $X$ be an infinite compact metric space,
and let $h \colon X \to X$ be a minimal homeomorphism.
Lin and the second author \cite{LP2010crossed} proved that there exists
an orbit breaking subalgebra $C^* (\mathbb{Z}, X, h)_y$ for some $y \in X$
with tracial rank zero whenever
the image of $K_0$ group of $C^* (\mathbb{Z}, X, h)$ is dense in
$\mathrm{Aff}\big( \mathrm{T} (C^* (\mathbb{Z}, X, h)) \big)$ and
$X$ has finite covering dimension.
Then by the large subalgebra approximation,
they proved that $C^* (\mathbb{Z}, X, h)$ has tracial rank zero.
Elliott and Niu \cite{EN2017Cast} proved that if $(X,h)$ has
mean dimension zero, then $C^* (\mathbb{Z}, X, h)$ is $\mathcal{Z}$-stable.
Archey and the second author \cite{AP2020permanence} proved that
if there is a continuous surjection from $X$ to the Cantor set,
then $C^* (\mathbb{Z}, X, h)$ has stable rank one.
By using a sharp embedding theorem in \cite{GT2020embedding},
the second author acturally proved in \cite{Phillips2016Cast} that
the radius comparison of $C^* (\mathbb{Z}, X, h)$ is upper bounded
by $1 + 2 \mdim(X,h)$, where $\mdim(X, h)$ is the mean topological dimension
of $(X,h)$ introduced in~\cite{LW2000mean}.
In \cite{ABP2020structure}, Archey, Buck and the second author constructed
large subalgebras of crossed products of $C(X, D)$ in which
$D$ is simple unital and studied the recursive subhomogeneous (RSH)
structure of the large subalgebra.
In \cite{Suzuki2020almost} Suzuki developed a local version of the
large subalgebra argument to compute the stable rank of
certain groupoid C*-algebras.

Many properties can be inherited from large subalgebras and vice versa.
It is proved in \cite{Phillips2014large} that for an infinite dimensional
simple unital C*-algebra $A$,
its stably large subalgebra $B$ must be simple and infinite dimensional.
The second author also studied other properties such as stably finiteness,
pure infiniteness,
the restriction map between trace state spaces,
and the relationship between their radii comparison.
In \cite{AP2020permanence}, Archey and the second author defined
centrally large subalgebras and studied the property of stable rank one.
In \cite{ABP2018centrally}, Archey, Buck and the second author studied
the property of tracial $\mathcal{Z}$-absorption
when the containing algebra is stably finite.
Fan, Zhao and the first author \cite{FFZ2019comparison} investigated
the properties of $m$-comparison of positive elements and
strong tracial $m$-comparison of positive elements.
Local weak comparison \cite{ZFF2020permanence},
almost divisibility \cite{FFZ2021inheritance}
and inheritance of tracial nuclear dimension for
centrally large subalgebras \cite{ZFF2020inheritance} were also studied.

Large subalgebra methods have flaws.
It is not clear how to explicitly construct a large subalgebra
of the crossed product when the coefficient algebra is simple,
or when the group is not integers.
There are some competing approaches
for obtaining regularity results of crossed products.
Hirshberg, Winter and Zacharias \cite{HWZ2015rokhlin}
developed the concept of Rokhlin dimension
for integers and for finite group actions on C*-algebras,
which is a higher-dimensional version of the
Rokhlin property \cite{Izumi2004finite_I, Izumi2004finite_II}.
They established permanence of
finite nuclear dimensionality and $\mathcal{Z}$-stability
for crossed products with finite Rokhlin dimension with commuting towers.
This method works well on both commutative and simple C*-algebras.
Gardella \cite{Gardella2017rokhlin} generalized this notion to compact group actions;
in this setting, more permanence properties were studied
in \cite{GHS2021rokhlin} by Gardella, Hirshberg and Santiago.
A nonunital version is developed in \cite{HP2015rokhlin} by Hirshberg
and the second author;
they also find a K-theoretic obstruction for actions of finite groups
to have finite Rokhlin dimension.
Recently, they \cite{HP2023values}
showed that any finite group admits
actions on simple AF algebras with unique trace
which have arbitary large but finite Rokhlin dimension with commuting towers.
Szab{\'o}, Wu and Zacharias \cite{SWZ2019rokhlin} extended the concept of
Rokhlin dimension to actions of residually finite groups,
and established a relation between Rokhlin dimension and
amenability dimension \cite{GWY2017dynamic}.
There are other generalizations of the Rokhlin property;
see \cite{FG2020weak} for some recent work,
and \cite{OP2006stable, Wang2013tracial,
	HO2013tracially, Archey2008crossed, MS2012stability,
	Hua2010tracial, Buck2010crossed}
for a related direction.

In the commutative case, more tools have been developed.
Szab{\'o} \cite{Szabo2015rokhlin} shows that free $\mathbb{Z}^m$-actions
on finite-dimensional $X$ satisfy a strengthened version of
the marker property,
which implies finite Rokhlin dimensionality.
Kerr \cite{Kerr2020dimension} introduced the notion of almost finiteness
and showed that almost finiteness implies $\mathcal{Z}$-stability
for free minimal action of infinite amenable groups on compact metric spaces.
See \cite{KS2020almost, LT2022almost} for related topics.
Niu \cite{Niu2019comparison_I, Niu2019comparison_II}
introduced two properties, the uniform Rokhlin property
and Cuntz comparison of open sets,
for topological dynamical systems with countable discrete amenable groups.
He proved that these two properties hold for minimal free
topological dynamical systems with $\mathbb{Z}^d$ action,
and showed that with these two properties,
the comparison radius of $C^* (\mathbb{Z}^d, X)$
is at most half of the mean dimension for any minimal free action.
See \cite{Niu2021Zstability, LN2020stable} for applications
to $\mathcal{Z}$-stability and stable rank one.

In this paper, we follow the main idea in \cite{Phillips2014large}.
We try to replace $C_0 (X \setminus \{y\}) \subset C(X)$
by a closed ideal $J$ of a noncommutative C*-algebra $A$,
and construct a large subalgebra of $C^{\ast}(\mathbb{Z}, A, \alpha)$
in Section~\ref{sec_large_subalg}.
As in the commutative case, this subalgebra is supposed to be
the C*-subalgebra generated by $A$ and $Ju$,
where $u$ is the canonical unitary element with respect to $1 \in \mathbb{Z}$.
We apply our result to $C(X)$-algebras and mapping tori in
Section~\ref{bundle_section} and Section~\ref{mapping_torus_sec}.

\section{Preliminaries}

The following two lemmas are elementary and well-known,
so we omit the proof.

\begin{lemma}\label{ideal_lattice_lem}
Let $C$ be a C*-algebra, and let $I$, $J$ be closed ideals of $C$.
\begin{enumerate}
\item\label{ideal_lattice_lem_1}
The closed ideal generated by $I$ and $J$ is $I+J$.
\item\label{ideal_lattice_lem_2}
$I \cap J = \{ab \colon a \in I , b \in J\}$.
\end{enumerate}
\end{lemma}

\begin{lemma}\label{unital_cpt_lem}
Suppose that $C$ is a unital C*-algebra.
Let $(I_n)_{n=1}^{\infty}$ be a sequence of closed ideals of $C$.
If the closed ideal generated by $\bigcup_{n=1}^{\infty} I_n$
is the whole algebra $C$,
then there exists $N \in \mathbb{Z}_{>0}$ such that
$\sum_{n=1}^{N} I_n = C$.
\end{lemma}

If $C$ is a C*-algebra,
we let $\Aut(C)$ denote the $*$-automorphism group of $C$.
We say that $C$ is $\alpha$-simple, or that $\alpha$ is minimal,
if the only $\alpha$-invariant closed ideals in $C$ are $\{0\}$ and $C$.
If $S \subset C$ is a subset,
we denote by $\langle S \rangle$ the closed ideal generated by $S$.

\begin{lemma}\label{minimal_equiv_lem}
Let $C$ be a unital C*-algebras and $\alpha \in \Aut(C)$.
The following are equivalent:
\begin{enumerate}
\item $C$ is $\alpha$-simple;
\item for any nonzero closed ideal $I$ of $C$, 
	if $I \subset \alpha (I)$, then $I = C$;
\item for any nonzero closed ideal $I$ of $C$,
	there exists $N \in \mathbb{Z}_{>0}$ such that
	\[
	\alpha^{-1} (I) + \alpha^{-2} (I) + \cdots + \alpha^{-N} (I) = C.
	\]
\end{enumerate}
\end{lemma}

\begin{proof}
(1)$\Rightarrow$(2).
Suppose by contradiction that $I$ is a nontrivial closed ideal
such that $\alpha^{-1}(I) \subset I$.
Let $J$ be the closed ideal generated
by $\bigcup_{n=0}^\infty \alpha^{n} (I)$.
We claim that $J$ is nontrivial.
In fact, if $J = C$, then we can apply Lemma \ref{unital_cpt_lem} to get
$C = J = \sum_{n=0}^N \alpha^n (I) = \alpha^N (I)$
for some $N \in \mathbb{Z}_{\geq 0}$.
It follows that $I = C$,
which is a contradiction.
This proves our claim.
Since $\alpha$ is a $*$-automorphism,
it is easy to see that for any subset $S \subset C$,
we have $\alpha(\langle S\rangle) = \langle \alpha(S)\rangle$.
So $\alpha(J) = \big\langle \bigcup_{n=1}^{\infty} \alpha^n (I)
	\big\rangle	\subset J$ and $\alpha^{-1}(J)
= \big\langle \alpha^{-1} (I)
	\cup \bigcup_{n=0}^{\infty} \alpha^{n} (I) \big\rangle \subset J$.
The $\alpha$-simplicity of $C$ implies that $C_1 = \{0\}$ or $C$,
which leads to the contradiction.
	
(2)$\Rightarrow$(3).
Let $I_1 = \big\langle \bigcup_{n=1}^\infty \alpha^{-n}(I) \big\rangle$.
According to Lemma~\ref{unital_cpt_lem} and Lemma~\ref{ideal_lattice_lem},
it suffices to show that $I_1 = C$,
which follows from $\alpha^{-1} (I_1)
= \big\langle \bigcup_{n=2}^\infty \alpha^{-n} (I)
	\big\rangle \subset I_1$.

(3)$\Rightarrow$(1).
Suppose that $I$ is a nonzero closed ideal of $C$ such that $\alpha(I)=I$.
Then $I = \alpha^{-1} (I) + \cdots + \alpha^{-N} (I) = C$ and so,
$\alpha$ is minimal.
\end{proof}

\begin{lemma}\label{lem24}
Let $C$ be a unital C*-algebra,
and let $I_1, I_2, \ldots , I_n$ be closed ideals of $C$.
If $I_1 + I_2 + \cdots + I_n = C$,
then there exists $r_1, r_2, \ldots, r_n \in C$ such that $r_k \in I_k$
and $0 \leq r_k \leq 1$ for $k = 1, 2, \ldots, n$,
and $r_1 + r_2 + \cdots + r_n =1$.
\end{lemma}

\begin{proof}
It suffices to prove that $C_+ = (I_1)_+ + \cdots + (I_n)_+$.
One can find a proof in \cite[Prop.~1.5.9]{Pedersen2018Cast}
by using a noncommutative version of the Riesz decomposition property.
\end{proof}

Let $A$ be a C*-algebra and $\alpha \in \Aut(A)$.
We say that $A$ is $\alpha$-prime,
if every pair of nonzero $\alpha$-invariant closed ideals of $A$
has nonzero intersection.
The notion of $\alpha$-primeness appeared in \cite{Murphy1996crossed}.
From the definition we immediately see that
if $A$ is prime or $\alpha$-simple, then $A$ is $\alpha$-prime.
\begin{lemma}\label{alpha_prime_lem}
Let $A$ be a C*-algebra and $\alpha \in \Aut(A)$.
The following are equivalent:
\begin{enumerate}
	\item $A$ is $\alpha$-prime;
	\item for any pair of closed ideals $I$ and $J$,
		there exists $n \in \mathbb{Z}$ such that
		$I \cap \alpha^n(J) \ne \{0\}$;
	\item for any $x, y \in A \setminus \{0\}$,
		there exists $n \in \mathbb{Z}$ such that
		$y A \alpha^n (x) \ne \{0\}$.
\end{enumerate}
\end{lemma}

\begin{proof}
(1)$\Rightarrow$(3).
Suppose by contradiction that there exists $x,y \in A \setminus \{0\}$
such that $yA\alpha^n(x) = \{0\}$ for all $n \in \mathbb{Z}$.
Then for any $m, n \in \mathbb{Z}$,
the algebraic ideal generated by $\alpha^m(y)$ is orthogonal to
the one generated by $\alpha^n(x)$.
For $z \in A$, let \[
I_z^\alpha = \Bigg\{ \sum_{j=1}^N a_j \alpha^{n_j}(z)b_j \colon
	N \in \mathbb{Z}_{>0} , \, a_j , b_j \in A , \,
	n_j \in \mathbb{Z} \Bigg\} \,.
\]
Then $\overline{I_z^\alpha}$ is the $\alpha$-invariant closed ideal
generated by $z$.
We get $\overline{I_y^\alpha} \cap \overline{I_x^\alpha} = \{0\}$
by Lemma~\ref{ideal_lattice_lem} (\ref{ideal_lattice_lem_2}),
which contradicts to the definition of $\alpha$-prime.

(2)$\Rightarrow$(1).
Trivial.
	
(3)$\Rightarrow$(2).
Taking any nonzero $y \in I$ and nonzero $x \in J$,
we get $n \in \mathbb{Z}$ and $z \in A$
such that $y z \alpha^n(x) \ne 0$.
So $y z \alpha^n(x) \in I \cap \alpha^n(J) \ne \{0\}$.
\end{proof}

\begin{lemma}\label{lem_JJ}
Let $I_1 , I_2 , \ldots , I_n$ and $J$ be closed ideals
of a unital C*-algebra $C$.
If $J + I_k = C$ for all $k=1, 2, \ldots, n$,
then $J + \bigcap_{k=1}^n I_k = C$.
\end{lemma}

\begin{proof}
Choose $x_k \in I_k$ for $k = 1, 2, \ldots, n$ such that
$1 - x_k \in J$.
Set $c = x_1 x_2 \cdots x_n$.
Then we have $c \in \bigcap_{k=1}^n I_k$
and $1 - c = (1 - x_1) + x_1 (1-x_2) + \cdots 
	+ x_1 x_2 \cdots x_{n-1} (1-x_n) \in J$.
\end{proof}

Let $A$ be a $\mathrm{C}^*$-algebra, and let $a,b \in A_+$.
Recall that $a$ is Cuntz subequivalent to $b$,
denoted by $a \precsim_A b$,
if there exists a sequences $(v_n)_{n=1}^{\infty}$ in $A$
such that $\| v_n b v_n^* - a\| \to 0$.
If $a \precsim_A b$ and $b \precsim_A a$,
we say that $a$ is Cuntz equivalent to $b$ and written by $a \sim_A b$.
Regard $M_n(A)$ as the upper left conner of $M_{n+1}(A)$,
and denote by $M_{\infty} (A)$ the algebraic direct limit
of the system $(M_n(A))_{n=1}^\infty$.
The Cuntz subequivalence can be defined similarly for
positive elements in $M_\infty(A)$;
see \cite[Def.~1.1]{Phillips2014large}.

The notions of large subalgebras, stably large subalgebras
and large subalgebras of crossed product type were introduced
in \cite[Sec.~4]{Phillips2014large}.
It was proved in \cite[Thm.~3.6]{AP2020permanence} that
large subalgebras of crossed product type of
stably finite C*-algebras are acturally centrally large.

\begin{definition}[Def.~4.9 of \cite{Phillips2014large}]
Let $D$ be an infinite dimensional simple separable unital $C^*$-algebra.
A C*-subalgebra $B \subset D$ is said to be
\emph{a large subalgebra of crossed product type}
if there exist a C*-subalgebra $A \subset B$
and a subset $G$ of the unitary group of $D$
such that:
\begin{enumerate}
	\item \begin{enumerate}
		\item $A$ contains the identity of $D$.
		\item $A$ and $G$ generate $D$ as a C*-algebra.
		\item $uAu^* \subset A$ and $u^* Au \subset A$ for all $u \in G$.
	\end{enumerate}
	\item For every $m \in \mathbb{Z}_{>0}$,
		$a_1 , \, a_2 , \, \ldots , \, a_m \in D$,
		$\epsilon>0$, $x \in D_+$ with $\|x\| = 1$,
		and $y \in B_+ \setminus \{0\}$,
		there exist $c_1 , \, c_2 , \, \ldots , \, c_m \in D$ and
		$g \in A$ such that:
	\begin{enumerate}
		\item $0 \leq g \leq 1$.
		\item For $j = 1,2,\ldots,m$
			we have $\|c_j - a_j\| < \epsilon$.
		\item For $j=1,2,\ldots,m$ we have $(1-g)c_j \in B$.
		\item $g \precsim_{B} y$ and $g \precsim_{D} y$.
		\item $\|(1-g)x(1-g)\| > 1 - \epsilon$.
	\end{enumerate}
\end{enumerate}
\end{definition}

The following lemma is well known and
a useful toolchain when dealing with the Cuntz subequivalence.
Lemma 1.4 of \cite{Phillips2014large} provides a guide
to (\ref{cuntz_comp_lem_1})-(\ref{cuntz_comp_lem_7}),
and the rest is from \cite[Cor.~1.6, Lem.~1.7]{Phillips2014large}.

\begin{lemma}\label{cuntz_comp_lem}
Let $A$ be a $\mathrm{C}^*$-algebra.
\begin{enumerate}
\item\label{cuntz_comp_lem_1}
	Let $a, b \in A_+$.
	If $a \in \overline{bAb}$, then $a \precsim_A b$.
\item\label{cuntz_comp_lem_2}
	Let $c \in A$. Then $c^* c \sim_A c c^*$.
\item\label{cuntz_comp_lem_3}
	Let $c \in A$ and let $\alpha > 0$.
	Then $(c^* c - \alpha)_+ \sim_A (c c^* - \alpha)_+$.
\item\label{cuntz_comp_lem_4}
	Let $a,b \in A_+$ and let $\epsilon > 0$.
	Then $\| a - b \| < \epsilon$
	implies $(a - \epsilon)_+ \precsim_A b$.
\item\label{cuntz_comp_lem_5}
	Let $a, b \in A_+$.
	Then $a + b \precsim_A a \oplus b$.
\item\label{cuntz_comp_lem_6}
	Let $a, b \in A_+$ and $ab = 0$.
	Then $a + b \sim_A a \oplus b$.
\item\label{cuntz_comp_lem_7}
	Let $a_1, \, a_2, \, b_1, \, b_2 \in A_+$,
	and suppose that $a_1 \precsim_A a_2$
	and $b_1 \precsim_A b_2$.
	Then $a_1 \oplus b_1 \precsim_A a_2 \oplus b_2$.
\item\label{cuntz_comp_lem_8}
	Let $\epsilon > 0$ and $\lambda \geq 0$.
	Let $a, b \in A$ satisfy $\| a - b\| < \epsilon$.
	Then $(a - \lambda - \epsilon)_+ \precsim_A (b - \lambda)_+$.
\item\label{cuntz_comp_lem_9}
	Let $a,b \in A$ satisfy $0 \leq a \leq b$.
	Let $\epsilon > 0$.
	Then $(a-\epsilon)_+ \precsim_A (b-\epsilon)_+$.
\end{enumerate}
\end{lemma}

The following lemma is also frequently used.
It is given in \cite[Cor.~2.5]{Phillips2014large};
see also \cite[Lem.~4.7]{Niu2019comparison_I} for another proof.

\begin{lemma}\label{cor_2_5}
Let $A$ be a simple C*-algebra which is not of type I.
Let $a \in A_+ \setminus \{0\}$,
and let $l \in \mathbb{Z}_{>0}$.
Then there exists $b_1 , \, b_2 , \, \ldots , \, b_l \in A_+ \setminus\{0\}$
such that $b_1 \sim_A b_2 \sim_A \cdots \sim_A b_l$,
such that $b_j b_k = 0$ for $j \ne k$,
and such that $b_1 + b_2 + \cdots + b_l \in \overline{aAa}$.
\end{lemma}

\section{Large subalgebras}\label{sec_large_subalg}

\begin{convention}\label{convention_coefficient_ideal}
Let $A$ be a C*-algebra, and let $J \subset A$ be a closed ideal.
Let $u$ denote the canonical unitary in $C^* (\mathbb{Z}, A, \alpha)$
such that $uf u^* = \alpha(f)$ for all $f \in A$.
Denote by $C_c(\mathbb{Z}, A, \alpha)$ the dense $*$-subalgebra
of $C^* (\mathbb{Z}, A, \alpha)$ given by all elements
of the form $\sum_{k=-N}^N f_k u^k$,
where $N \in \mathbb{Z}_{\geq 0}$ and $f_k \in A$ for all $k$.
Denote by $C^{\ast}(\mathbb{Z}, A, \alpha)_J$ the C*-subalgebra
of $C^* (\mathbb{Z}, A, \alpha)$ generated by $A$ and $Ju$.
Let $E \colon C^* (\mathbb{Z}, A, \alpha) \to A$ be the standard
condition expectation onto $A$.
For $n \in \mathbb{Z}$, set \[
J_n = \begin{cases}
	\bigcap_{i=0}^{n-1} \alpha^i (J) \qquad & n>0 \\
	A \qquad & n=0 \\
	\bigcap_{i=1}^{|n|} \alpha^{-i} (J) \qquad & n<0 .
\end{cases}
\]
\end{convention}

\begin{proposition}\label{piece_1}
Adopt Convention~\ref{convention_coefficient_ideal}.
Then \[
	C^* (\mathbb{Z}, A, \alpha)_J
		= \big\{ x \in C^* (\mathbb{Z}, A, \alpha)
		\colon E(xu^{-n}) \in J_n \enspace
		\text{for all} \enspace n \in \mathbb{Z} \big\}
	\]
and
$C^* (\mathbb{Z}, A, \alpha)_J \cap C_c(\mathbb{Z}, A, \alpha)$
is dense in $C^* (\mathbb{Z}, A, \alpha)_J$.
\end{proposition}

\begin{proof}
Let $B = \{ x \in C^* (\mathbb{Z}, A, \alpha) \colon
	E(xu^{-n}) \in J_n \enspace \text{for all}
	\enspace n \in \mathbb{Z}\}$ and
$B_0 = B \cap C_c(\mathbb{Z}, A, \alpha)$.
We claim that $B_0$ is dense in $B$.
For any $b \in B$, we have $E(bu^{-k}) \in J_{k}$
for all $k \in \mathbb{Z}$.
By Fej{\'e}r's Theorem \cite[Thm.~VIII.2.2]{Davidson1996Cast},
\[
	\sigma_n (b) = \sum_{|k| \leq n} \Big( 1- \frac{|k|}{n+1} \Big)
	E(bu^{-k})u^k \to b \quad \text{in norm.}
\]
Since $\sigma_n (b) \in B_0$ for all $n \in \mathbb{Z}$,
the claim follows.
It remains to show that $B = C^* (\mathbb{Z}, A, \alpha)_J$.

Next, we prove that $B_0$ is a $*$-algebra.
To prove this,
we claim that $\alpha^{-n}(J_n) = J_{-n}$ and
$J_n \cap \alpha^n (J_m) \subset J_{n+m}$ for all $n,m \in \mathbb{Z}$.
For $n=0$ the claim is trivial. 
For $n>0$, \begin{align*}
\alpha^{-n}(J_n) 
& =\alpha^{-n} \big(J \cap \alpha(J) \cap
	\cdots \cap \alpha^{n-1}(J)\big) \\
& = \alpha^{-n}(J) \cap \alpha^{-(n-1)} (J) \cap
	\cdots \cap \alpha^{-1} (J) \\
& = J_{-n} .
\end{align*}
Acting $\alpha^n$ on both sides gives the case $n < 0$.
	
To see that $J_n \cap \alpha^n (J_m) \subset J_{n+m}$
for all $n,m \in \mathbb{Z}$,
we discuss the combinations of signs of $m,n$ and $m+n$.
The case $n = 0$ or $m = 0$ is trivial.
A routine computation gives the case $n,m>0$ and the case $n,m<0$.
In fact, under these circumstances, $J_n \cap \alpha^n (J_m) = J_{n+m}$.
When $m<0<n$ and $n+m \geq 0$, we have $m < 0 \leq n+m < n$.
It is obviously that $J_n \cap \alpha^n(J_m) \subset J_n \subset J_{n+m}$.
When $m < 0 < n$ and $n + m < 0$,
we have $m < n+m < 0 <n$.
Then \begin{align*}
J_n \cap \alpha^n (J_m)
& = \big( J \cap \alpha (J) \cap \cdots \cap \alpha^{n-1}(J) \big)
	\cap \alpha^n \big( \alpha^{-1} (J) \cap
	\cdots \cap \alpha^{-|m|} (J) \big) \\
& \subset \alpha^{n+m}(J) \cap \cdots \cap \alpha^{-1}(J) \\
& = J_{n+m} .
\end{align*}
A similar argument gives the case $n < 0 < m$.
	
	To prove that $B_0$ is $*$-algebra,
	it remains to show that if $b \in J_n$ and $c \in J_m$,
	then $(bu^n)^* \in B_0$ and $(bu^n) (cu^m) \in B_0$.
	Since $J$ is closed under the $*$ operation,
	it follows that $J_n$ is self adjoint.
	Note that
	$(bu^n)^* = u^{-n}b^* = u^{-n} b^* u^n u^{-n} = \alpha^{-n}(b^*) u^{-n}$.
	The facts $b^* \in J_n$ and $\alpha^{-n}(J_n) = J_{-n}$
	implie $\alpha^{-n}(b^*) \in J_{-n}$.
	So $(bu^n)^* \in B_0$.
	Consider $(bu^n)(cu^m) = b(u^ncu^{-n})u^{n+m} = b \alpha^{n}(c) u^{n+m}$.
	As $c \in J_m$, $\alpha^n(c) \in \alpha^n(J_m)$.
	So $b\alpha^n(c) \in J_n \cap \alpha^n(J_m)$.
	The previous claim implies that $b\alpha^n(c) \in J_{n+m}$.
	Hence $(bu^n)(cu^m) \in B_0$.
	
	Since $B_0$ contains $A$ and $Ju$,
	it follows that
	$C^* (\mathbb{Z}, A, \alpha)_J \subset \overline{B_0} = B$.
	It remains to show that $B_0 \subset C^* (\mathbb{Z}, A, \alpha)_J$,
	or equivalently,
	to show that $J_n u^n \subset C^* (\mathbb{Z}, A, \alpha)_J$
	for all $n \in \mathbb{Z}$.
	For $n=0$ this is trivial. 
	Suppose that $n>0$. Let $b \in J_n$.
	Our goal is to prove that $bu^n \in C^* (\mathbb{Z}, A, \alpha)_J$.
	Without loss of generality, one may assume that $b$ is positive.
	Set $c = b^{1/n}$. Then \begin{align*}
	bu^n & = c^n u^n
	= (c u) [(u^{-1} c u)u][(u^{-2}c u^2)u]
		\cdots [(u^{-(n-1)}c u^{n-1})u] \\
	& = (c u) [\alpha^{-1}(c)u] [\alpha^{-2}(c)u]
		\cdots [\alpha^{-(n-1)}(c)u].
	\end{align*}
	Since $c \in J_n$, one has $\{\alpha^{-k}(c)\}_{k=0}^{n-1} \subset J$.
	It follows that $bu^n$ is the finite product of elmemts of $Ju$.
	Hence $bu^n \in C^* (\mathbb{Z}, A, \alpha)_J$.
	Finally, suppose $n<0$, and let $b \in J_n$.
	It follows from $\alpha^{-n} (J_n) = J_{-n}$
	that $\alpha^{-n}(b^*) \in J_{-n}$.
	We therefore get
	\[
		bu^n = (u^{-n}b^*)^* = (u^{-n} b^* u^n u^{-n})^*
		= (\alpha^{-n}(b^*) u^{-n})^* \in (J_{-n} u^{-n})^*
		\subset C^* (\mathbb{Z}, A, \alpha)_J .
	\]
	This proves $J_n u^n \subset C^* (\mathbb{Z}, A, \alpha)_J$
	for all $n \in \mathbb{Z}$.
	So $B_0 \subset C^* (\mathbb{Z}, A, \alpha)_J$,
	and $\overline{B_0} = B \subset C^* (\mathbb{Z}, A, \alpha)_J$.
	Combining this result with $C^* (\mathbb{Z}, A, \alpha)_J \subset B$,
	we get $C^* (\mathbb{Z}, A, \alpha)_J = B$.
\end{proof}

\begin{corollary}\label{corollary_of_piece_1}
	Adopt Convention~\ref{convention_coefficient_ideal}.
	Suppose that $(J^{(n)})$ is an increasing sequence of
	closed ideals of $A$ such that $J = \overline{\bigcup_{n} J^{(n)}}$.
	Then \[
		C^* (\mathbb{Z}, A, \alpha)_{J}
		= \overline{\bigcup_{n} C^* (\mathbb{Z}, A, \alpha)_{J^{(n)}}}
		= \varinjlim_{n} C^{\ast}(\mathbb{Z}, A, \alpha)_{J^{(n)}} .
	\]
	If there exists $N > 0$ such that $J_N = 0$,
	then $C^{\ast}(\mathbb{Z}, A, \alpha)_{J}$ has
	the Banach space decomposition \[
		C^{\ast}(\mathbb{Z}, A, \alpha)_{J}
		= \bigoplus_{k = -(N-1)}^{N-1} J_k u^k
		= \left(\bigoplus_{k=1}^{N-1} u^{-k}J_k \right)
			\oplus A \oplus \left( \bigoplus_{k=1}^{N-1} J_k u^k \right) \,.
	\]
\end{corollary}

\begin{proof}
	For any $\epsilon > 0$ and any $a \in J$,
	there exists $n$ and $b \in J^{(n)}$ such that $\|a - b\| < \epsilon$.
	Therefore $\|au - bu\| < \epsilon$.
	Since $J^{(n)}$ is an increasing sequence,
	the connecting map
	$C^* (A , J^{(n)}u) \to C^* (A, J^{(n+1)}u)$ is a unital inclusion.
	So $C^* (\mathbb{Z}, A, \alpha)_{J}
	= \varinjlim_{n} C^*(\mathbb{Z}, A, \alpha)_{J^{(n)}}$.
	
	For the Banach space decomposition,
	note that $u^{-k} J_{k} u^{k} = \alpha^{-k}(J_{k}) = J_{-k}$,
	which gives $J_k u^k = u^{k} J_{-k}$ for $k < 0$.
	So Proposition~\ref{piece_1} gives the algebraic direct sum decomposition
	of $C^{\ast}(\mathbb{Z}, A, \alpha)_{J}$,
	and it is a direct sum decomposition of Banach spaces
	by the Open Mapping Theorem.
\end{proof}

The following lemma is awkward.
It is merely a mechanical analog of the operation on open neigoborhoods 
in an infinite minimal topological dynamical system $(X, h)$.
It is helpful to think of $C$ as $C(X)$ and $J$ as $C_0(X \setminus \{y\})$
for some fixed $y \in X$.
Then $\alpha^{-k}(f) \in J$ is an analog of $f(h^k(y)) = 0$.

\begin{lemma}\label{J_prop}
	Suppose that $C$ is a unital C*-algebra and $\alpha \in \Aut(C)$.
	Let $J$ be a nontrivial closed ideal of $C$,
	such that for every $k \in \mathbb{Z} \setminus \{0\}$,
	there exists $f \in C$ with $0 \leq f \leq 1$ such that
	\begin{enumerate}
		\item $\alpha^{-k}(f) \in J$, $1 - f \in J$, and
		\item $f C \alpha^{-k}(f) = 0$.
	\end{enumerate}
	Then for any $n \in \mathbb{Z}_{>0}$,
	there exists a closed ideal $I \subset \bigcap_{k=1}^{n} \alpha^{-k} (J)$
	such that $\alpha^{-n}(I), \alpha^{-(n-1)}(I), \ldots, \alpha^{n}(I)$
	are pairwise orthogonal and $J + I = C$.
\end{lemma}

\begin{proof}
	From the condition we know that for
	every $k \in \mathbb{Z} \setminus \{0\}$,
	there is a nontrivial closed ideal $\tilde{I}^{(k)}$ of $\alpha^k(J)$
	generated by the corresponding $f$ such that
	$\tilde{I}^{(k)} + J = C$ and
	$\tilde{I}^{(k)}$ is orthogonal to $\alpha^k(\tilde{I}^{(k)})$.

	We prove the conclusion by induction on $n$.
	Let $n = 1$. 
	Let $k=-1,1,2$ and get $\tilde{I}^{(-1)} , \tilde{I}^{(1)}$
	and $\tilde{I}^{(2)}$, respectively.
	Then $I_1 = \tilde{I}^{(-1)} \cap \tilde{I}^{(1)} \cap \tilde{I}^{(2)}$
	is a nontrivial closed ideal such that
	$J + I_1 = C$ by Lemma~\ref{lem_JJ},
	and $\{\alpha^{-1}(I_1) , I_1 , \alpha(I_1)\}$ are pairwise orthogonal.
	We also have $I_1 \subset \alpha^{-1}(J)$
	by $\tilde{I}^{(-1)} \subset \alpha^{-1}(J)$.
		
	Suppose that the proposition is true for $1, 2, \cdots, n$.
	One may find closed ideal $I_n \subset \bigcap_{i=1}^n \alpha^{-i}(J)$
	such that $\{\alpha^k(I_n)\}_{k=-n}^n$ are pairwise orthogonal
	and $J + I_n = C$.
	Applying the condition for $k=-(n+1), 2n+1, 2n+2$,
	we get closed ideals $\tilde{I}^{(-n-1)}$,
	$\tilde{I}^{(2n+1)}$ and $\tilde{I}^{(2n+2)}$, respectively.
	Let $I_{n+1} = I_n \cap \tilde{I}^{(-n-1)}
	\cap \tilde{I}^{(2n+1)} \cap \tilde{I}^{(2n+2)}$.
	Then $I_{n+1} \subset I_n \cap \tilde{I}^{(-n-1)}
	\subset \bigcap_{k=1}^{n+1} \alpha^{-k} (J)$,
	$J + I_{n+1} = C$ and $\{\alpha^k(I_{n+1})\}_{k=-(n+1)}^{n+1}$
	are pairwise orthogonal.
\end{proof}

The following proposition is an analog of Lemma~7.7 in \cite{Phillips2014large}.

\begin{lemma}\label{piece_2}
	Let $A$ be a unital C*-algebra,
	and let $C \subset A$ be a unital C*-subalgebra.
	Suppose that $\alpha$ is a minimal $*$-automorphism of $A$,
	such that $\alpha|_C \in \Aut(C)$ and $C$ is $\alpha|_C$-simple.
	Let $J$ be a nontrivial closed ideal of $A$ such that
	$J|_C = J \cap C$ is a nontrivial closed ideal of $C$
	satisfying the hypotheses in Lemma~\ref{J_prop}.
	Set $B = C^* (\mathbb{Z}, A, \alpha)_{J}$.
	Then for any nonzero closed ideal $I$ of $C$ and any $k \in \mathbb{Z}$,
	there exists nonzero positive contractions $r \in C$ and $s \in I$
	such that $r \precsim_{B} s $ and $1-r \in \alpha^k (J)$.
\end{lemma}

\begin{proof}
	We first prove this when $k=0$.
	
	By Lemma~\ref{minimal_equiv_lem},
	there exists $n \in \mathbb{Z}_{>0}$ such that
	$\sum_{k=1}^n \alpha^{-k} (I) = C$.
	Apply Lemma~\ref{J_prop} and obtain a nontrivial closed ideal
	$I' \subset \bigcap_{j=1}^{n} \alpha^{-j}(J |_C)$
	such that $\alpha^{-n}(I'), \alpha^{-(n-1)}(I'), \ldots, \alpha^{n}(I')$
	are pairwise orthogonal and $J|_C + I' = C$.
	So
	$
		J|_C + I' = J|_C + \sum_{i=1}^n \alpha^{-i}(I)\cap I' = C .
	$
	Use Lemma~\ref{lem24} to choose
	$r_i \in \alpha^{-i}(I) \cap I'$ for $i=1, 2, \ldots, n$
	with $0 \leq r_i \leq 1$
	such that $1 - \sum_{i=1}^n r_i \in J|_C$.
	Then $\alpha(r_1), \alpha^{2}(r_2), \ldots, \alpha^{n}(r_n)$ are
	pairwise orthogonal in $I$.
	Since $r_i \in \bigcap_{j=1}^{n} \alpha^{-j}(J)$,
	we have $r_i^{1/2} u^{i} \in B$ by Proposition~\ref{piece_1}.
	Therefore, by applying Lemma~\ref{cuntz_comp_lem} \eqref{cuntz_comp_lem_2}
	to $B$, we have \[
		\alpha^i(r_i) = (r_i^{1/2} u^{-i})^* (r_i^{1/2} u^{-i})
		\sim_{B} (r_i^{1/2} u^{-i}) (r_i^{1/2} u^{-i})^* = r_i .
	\]
	Let $r = \sum_{i=1}^n r_i$. Then $1-r \in J|_C$.
	Set $s = \sum_{i=1}^n \alpha^i (r_i)$. Then $s \in I$ and
	\[
		r = \sum_{i=1}^N r_i \precsim_{B} \bigoplus_{i=1}^N r_i
		\sim_{B} \bigoplus_{i=1}^N \alpha^i(r_i)
		\sim_{B} \sum_{i=1}^N \alpha^i (r_i) = s
	\]
	due to Lemma~\ref{cuntz_comp_lem}
	\eqref{cuntz_comp_lem_5} and \eqref{cuntz_comp_lem_6}.
	This completes the proof for $k=0$.
	
	Now suppose that $k \ne 0$.
	Find $r_0 \precsim_{B} s \in I$ for the case $k=0$
	such that $1 - r_0 \in J|_C$.
	Using Lemma~\ref{J_prop} for $j=-1, -2, \ldots, -k$,
	one may find $f_j \in C$ with $0 \leq f_j \leq 1$
	satisfying $1 - f_j \in J|_C$ and $f_j \in \alpha^j (J|_C)$.
	Let $h = f_{-k} \cdots f_{-1} r_0 f_{-1} \cdots f_{-k}$.
	Then $h \in \bigcap_{j=1}^k \alpha^{-j}(J)$, $0 \leq h \precsim_{C} r_0$
	and $1-h \in J|_C$.
	Therefore $h^{1/2}u^{-k} \in B$ by Proposition~\ref{piece_1}.
	Set $r = \alpha^k(h)$. Then $1-r \in \alpha^k(J|_C)$ and
	\[
		r = u^k h u^{-k} = (h^{1/2}u^{-k})^* (h^{1/2}u^{-k})
		\sim_{B} (h^{1/2}u^{-k}) (h^{1/2}u^{-k})^* = h \precsim_{B} r_0 \precsim_{B} s .
	\]
	This completes the proof.
\end{proof}

The following definition originally appeared in
\cite[Lem.~3.2]{Kishimoto1981outer} as a conclusion;
we call it \emph{Kishimoto's condition} in \cite[Def.~10.4.20]{GKPT2018Crossed}.

\begin{definition}\label{def_Kishimoto}
	Let $\alpha \colon G \to \Aut(A)$ be an action
	of a discrete group $G$ on a C*-algebra $A$.
	We say that $\alpha$ satisfies \emph{Kishimoto's condition} if,
	for every positive element $x \in A$ with $\|x\| = 1$,
	every finite set $F \subset G$,
	every finite set $S \subset A$, and every $\epsilon > 0$,
	there is a positive element $c \in A$ with $\|c\| = 1$ such that
	$\|cxc\| > 1 - \epsilon$ and $\|{cb\alpha_g(c)}\| < \epsilon$
	for all $g \in F$ and $b \in S$.
\end{definition}

The following theorem is contained in~\cite[Thm.~3.1]{Kishimoto1981outer};
see the proof of \cite[Thm.~10.4.22]{GKPT2018Crossed}
for arbitary discrete group.

\begin{theorem}\label{Kishimoti1981thm}
	Let $A$ be a C*-algebra,
	and let $\alpha \colon \mathbb{Z} \to \Aut(A)$ be
	an action of $\mathbb{Z}$ on $A$.
	Suppose that $A$ is $\alpha$-simple and
	$\alpha$ satisfies Kishimoto's condition.
	Then $\mathrm{C}^*(\mathbb{Z},A,\alpha)$ is simple.
\end{theorem}

The following lemma is similar to
\cite[Lem.~7.9]{Phillips2014large} and \cite[Lem.~4.2]{Niu2019comparison_I},
but in the noncommutative setting we use
Kishimoto's condition to replace the freeness condition.

\begin{lemma}\label{Kishimoto_prime_lem}
	Let $A$ be a C*-algebra, and let $\alpha \in \Aut(A)$.
	Suppose that $\alpha$ satisfies Kishimoto's condition.
	Let $B \subset \mathrm{C}^* (\mathbb{Z},A,\alpha)$ be a $C^*$-subalgebra
	such that $A \subset B$
	and $B \cap C_c(\mathbb{Z}, A, \alpha)$ is dense in $B$.
	Let $b \in B_{+} \setminus \{0\}$.
	Then there exists $a \in A_{+} \setminus \{0\}$
	such that $a \precsim_{B} b$.
\end{lemma}

\begin{proof}
	Without lose of generality, one may assume that $\|b\| \leq 1$.
	Recall that the condition expectation
	$E : \mathrm{C}^* (\mathbb{Z},A,\alpha) \to A$
	is faithful and contractive.
	Set $\epsilon = \|\mathrm{E}(b)\| / 8$.
	Then $0 < \epsilon \leq 1/8$.
	Choose $c \in B\cap C_c(\mathbb{Z}, A, \alpha)$ such that
	$\| c - b^{{1}/{2}}\| < \epsilon$ and $\| c\| \leq 1$.
	Then $\| c^*c - b\| < 2\epsilon$ and $\| cc^* - b\| < 2\epsilon$.
	Suppose that $c^* c = \sum_{k=-n}^n a_k u^k$,
	where $n \in \mathbb{Z}_{>0}$
	and $a_k \in A$ for $k = -n , \ldots, n$.
	Note that $a_0 = E(c^* c) \in A_{+} \setminus \{0\}$.
	Apply Kishimoto's condition (see Definition \ref{def_Kishimoto}) with
	$\|a_0\|^{-1} a_0$ in place of $x$,
	with $\{-n , \ldots, n\} \setminus \{0\}$ in place of $F$,
	with $\{ a_k : 0 < |k| \leq n\}$ in place of $S$,
	and with ${\epsilon}/{(2n)}$ in place of $\epsilon$,
	getting $f \in A_{+}$ with $\|f\|=1$ such that
	\[
	\|f a_0 f\| > \left( 1 - \frac{\epsilon}{2n} \right) \|a_0\|
	\]
	and
	\[
	\|fc^*cf - f a_0 f\|
	\leq \sum_{0 < |k| \leq n} \| f a_k u^k f \|
	= \sum_{0 < |k| \leq n} \| f a_k \alpha^k (f) \|
	< \frac{2n\epsilon}{2n} = \epsilon .
	\]
	Let $x = f a_0 f$. Then $x \in A_{+}$.
	Since $\|c^*c - b\| < 2\epsilon$,
	we have $\|a_0\| \geq \| E(b)\| - 2\epsilon$.
	Thus
	\[
	\|x\| > \left( 1 - \frac{\epsilon}{2n} \right) \|a_0\|
	\geq \left( 1 - \frac{\epsilon}{2n} \right)
		\left(\|E(b)| - 2\epsilon\right)
	= \left(1 - \frac{\epsilon}{2n}\right) 6\epsilon > 5\epsilon .
	\]
	So $(x - 3\epsilon)_{+}$ is a nonzero positive element in $A$.
	Using Lemma~\ref{cuntz_comp_lem} \eqref{cuntz_comp_lem_8}
	at the first step,
	\eqref{cuntz_comp_lem_3} at the second step,
	\eqref{cuntz_comp_lem_9} at the third step
	and \eqref{cuntz_comp_lem_4} at the last step,
	we then have \[
	(x - 3\epsilon)_{+} \precsim_B (fc^*cf - 2\epsilon)_{+}
	\sim_B (cf^2 c^* - 2\epsilon)_{+} \precsim_B (cc^* - 2\epsilon)_{+}
	\precsim_B b .
	\]
	Thus $a = (x-3\epsilon)_{+}$ satisfies our requirements.
\end{proof}

\begin{corollary}\label{piece_3}
	Let $A$ be a C*-algebra, and let $\alpha \in \Aut(A)$.
	Suppose that $\alpha$ satisfies Kishimoto's condition
	and $A$ is $\alpha$-prime.
	Set $B = C^* (\mathbb{Z}, A, \alpha)_J$.
	Then for any $x \in C^* (\mathbb{Z}, A, \alpha)_{+} \setminus \{0\}$
	and $y \in B_{+} \setminus \{0\}$,
	there exists $z \in A_{+} \setminus \{0\}$
	such that $z \precsim_{C^* (\mathbb{Z}, A, \alpha)} x$
	and $z \precsim_{B} y$.
\end{corollary}

\begin{proof}
	Applying Lemma~\ref{Kishimoto_prime_lem} to $x$
	with $C^*(\mathbb{Z}, A, \alpha)$ in place of $B$,
	and to $y$ with $B$ as given,
	we see that it suffices to proof the corollary
	for $x, y \in A_+ \setminus \{0\}$.
	Since $A$ is $\alpha$-prime,
	it follows from Lemma~\ref{alpha_prime_lem} that
	there exists $n \in \mathbb{Z}$ and $d \in A$
	such that $ y^{1/2} d \alpha^n(x^{1/2}) \ne 0$.
	Without loss of generality, one may assume that
	$\| x\| \leq 1$ and $ \|d\|\leq 1$.
	Set $z = y^{1/2} d \alpha^n(x) d^* y^{1/2}$.
	Then $z \in A_{+} \setminus \{0\}$ and $z \leq y$.
	Thus $z \precsim_{B} y$ by Lemma~\ref{cuntz_comp_lem}
	\eqref{cuntz_comp_lem_1}. On the other hand,
	$z = (y^{1/2} b u^n) x (y^{1/2} b u^n)^* 
	\precsim_{C^* (\mathbb{Z}, A, \alpha)} x$.
\end{proof}

\begin{theorem}\label{main_thm}
	Suppose that $A$ is a unital C*-algebra and
	$C \subset A$ is a unital C*-subalgebra.
	Let $\alpha \in \Aut(A)$ be minimal such that $\alpha|_C \in \Aut(C)$,
	and suppose that $\alpha$ satisfies the Kishimoto's condition.
	Let $J$ be a nontrivial closed ideal of $A$ such that
	$J|_C = J \cap C$ is the nontrivial closed ideal of $C$
	as in Lemma~\ref{J_prop}.
	Set $B = C^* (\mathbb{Z}, A, \alpha)_{J}$.
	Assume that $C$ is $\alpha|_C$-simple with the following properties:
	\begin{enumerate}
		\item\label{main_thm_assumption_1}
		for any $n \in \mathbb{Z}_{> 0}$,
		any $c \in C_+ \setminus \{0\}$,
		there exists $n$ orthogonal nonzero positive element
		$e_1 , \ldots ,e_n$ in $\overline{cCc}$;
		\item\label{main_thm_assumption_2}
		for any $a \in A_+ \setminus \{0\}$,
		there exists $e \in C_+ \setminus \{0\}$ such that
		$e \precsim_{B} a$.
	\end{enumerate}
	If $C^* (\mathbb{Z}, A, \alpha)$ is finite, 
	then $B$ is a large subalgebra
	of $C^* (\mathbb{Z}, A, \alpha)$ of crossed product type.
\end{theorem}

We verify the conditions of the following lemma
to prove Theorem~\ref{main_thm}.
It is included in Proposition 4.11 of \cite{Phillips2014large}.

\begin{lemma}\label{proposition_411}
	Let $D$ be an infinite dimensional simple unital C*-algebra,
	and let $B \subset D$ be a unital C*-subalgebra.
	Let $A \subset B$ be a C*-subalgebra,
	and let $G$ be a subset of the unitary group of $D$.
	Assume that the following conditions are satisfied:
	\begin{enumerate}
		\item $D$ is finite.
		\item \begin{enumerate}
			\item $A$ contains the identity of $D$.
			\item $A$ and $G$ generate $D$ as a C*-algebra.
			\item $uAu^* \subset A$ and $u^* Au \subset A$ for all $u \in G$.
			\item For any $x \in D_+ \setminus \{0\}$
				and $y \in B_+ \setminus \{0\}$,
				there exists $z \in B_+ \setminus \{0\}$
				such that $z \precsim_{D} x$ and $z \precsim_B y$.
		\end{enumerate}
		\item For every $m \in \mathbb{Z}_{>0}$,
			$a_1 , a_2 , \ldots , a_m \in D$,
			$\epsilon>0$ and $y \in B_+ \setminus \{0\}$,
			there exists $c_1, c_2, \ldots , c_m \in D$
			and $g \in A$ such that:
			\begin{enumerate}
				\item $0 \leq g \leq 1$.
				\item For $j=1,2,\ldots,m$
					we have $\| c_j - a_j \| < \epsilon$.
				\item For $j=1,2,\ldots,m$
					we have $(1-g) c_j \in B$.
				\item $g \precsim_{B} y$.
			\end{enumerate}
	\end{enumerate}
	Then $B$ is a large subalgebra of $D$ of crossed product type.
\end{lemma}

\begin{proof}[Proof of Theorem~\ref{main_thm}]
	Before we verify conditions listed in Lemma~\ref{proposition_411},
	first note that the property \eqref{main_thm_assumption_1} on $C$
	in Theorem~\ref{main_thm} implies that $C$ is infinite dimensional.
	So $C^* (\mathbb{Z}, A, \alpha)$ is an infinite dimensional unital C*-algebra,
	and $C^* (\mathbb{Z}, A, \alpha)$ is simple by Theorem \ref{Kishimoti1981thm}.

	Conditions (1) and (2)(a,b,c) are trivial.
	Condition (2)(d) follows by Corollary~\ref{piece_3}.
	Let $m \in \mathbb{Z}_{>0}$,
	let $a_1 , a_2, \ldots , a_m \in C^* (\mathbb{Z}, A, \alpha)$,
	let $\epsilon>0$,
	and let $y \in B_+ \setminus \{0\}$.
	Pick $c_j \in C_c(\mathbb{Z}, A, \alpha)$ such that
	$\| c_j - a_j\| < \epsilon$ for $j=1, 2, \cdots, m$.
	One may write
	\[
	 c_j = \sum_{k = -n}^n d_{j,k} u^k ,
	 \quad \text{where} \enspace d_{j,k} \in A \enspace \text{for} \enspace
	 j = 1, 2, \ldots, m \enspace \text{and} \enspace k = -n , \ldots, n .
	\]
	Using Corllary~\ref{piece_3} to	choose
	$a \in A_{+} \setminus \{0\}$ such that $a \precsim_{B} y$.
	By our assumption \eqref{main_thm_assumption_1} and
	\eqref{main_thm_assumption_2} on $C$ and
	Lemma~\ref{cuntz_comp_lem} \eqref{cuntz_comp_lem_6},
	there exists nonzero positive elements $e_k \in C$
	for $k = -n , -n+1, \ldots , n$,
	such that $e_{-n} \oplus e_{-n+1} \oplus \cdots \oplus e_n \precsim_{B} a$.
	For each $k=-n, -n+1, \cdots, n$,
	let $I_k = \{ c \in C : c^* c \precsim_{C} e_k\}$.
	Then $I_k$ is a closed ideal of $C$ by \cite[Lem.~3.3]{FLL2022tracial}.
	Apply Proposition~\ref{piece_2} to $I_k$ and $k$,
	getting positive contractions $r_k , s_k \in C$ such that
	$r_k \precsim_{B} s_k \in I_k$ and $1-r_k \in \alpha^k(J|_C)$.
	Using Lemma~\ref{cuntz_comp_lem} \eqref{cuntz_comp_lem_7}
	at the first two steps, we get
	\begin{equation}\label{equ_MainThm}
		\bigoplus_{k=-n}^n r_k \precsim_B \bigoplus_{k=-n}^n s_k
		\precsim_B \bigoplus_{k=-n}^n e_k 
		\precsim_B a \precsim_B y .
	\end{equation}
	By Lemma~\ref{J_prop} and Lemma~\ref{lem24},
	there exists $h \in C$ with $0 \leq h \leq 1$ such that
	$\{\alpha^k(h)\}_{k=-n}^n$ are pairwise orthogonal
	and $1-h \in J|_C$.
	Set \[
		f = \alpha^{-n} (r_{n}) \alpha^{-n+1}(r_{n-1})
			\cdots \alpha^{n}(r_{-n})h
		\quad \text{and} \quad g_0 = f^* f \,.\]
	Let $g_k = \alpha^{k} (g_0)$.
	Then $g_k \precsim_B r_k$ by \cite[Prop.~2.4]{Rordam1992structure_II}
	and $g_k g_j = 0$ for $k \ne j$.
	Let $g = \sum_{k=-n}^n g_k$.
	Then $0 \leq g \leq 1$.
	Using Lemma~\ref{cuntz_comp_lem} \eqref{cuntz_comp_lem_5} at the first step, 
	Lemma~\ref{cuntz_comp_lem} \eqref{cuntz_comp_lem_7} at the second step
	and \eqref{equ_MainThm} at the last step,
	we get
	\[
		g \precsim_B \bigoplus_{k=-n}^n g_k 
		\precsim_B \bigoplus_{k=-n}^n r_k
		\precsim_B y .
	\]
	Note that $1 - h$ and $1 - \alpha^{-k}(r_k) \in J$
	for $k=-n, -n+1, \ldots, n$.
	So $1 - g_0 \in J$ and then $1 - g_k \in \alpha^k(J)$ for all $k$.
	Therefore by orthogonality of $(g_k)$, we have
	\[
	1 - g = (1 - g_{-n}) (1 - g_{-n + 1}) \cdots (1 - g_{n})
	\in \bigcap_{k=-n}^n \alpha^k(J) .
	\]
	By Proposition~\ref{piece_1},
	we have $(1-g)c_j \in B$ for all $j = 1, 2, \ldots, m$.
	This completes the verification of condition (3)
	and the proof of the theorem.
\end{proof}

\section{$C(X)$-algebras}\label{bundle_section}

Let $X$ be a compact Hausdorff space.
For a C*-algebra $B$,
we denote by $Z(B)$ the center of $B$.
Recall that a $C(X)$-algebra is a C*-algebra $A$
endowed with a unital $*$-homomorphism $\eta \colon C(X) \to Z(M(A))$,
where $M(A)$ is the multiplier algebra of $A$.
We shall refer to the map $\eta$ as the \emph{structure map}.

For any closed set $Y \subset X$,
we identify $C_0(X \setminus Y)$ with the ideal of
continuous functions vanishing on $Y$.
Then $\eta(C_0(X \setminus Y))A$ is a closed ideal of $A$
by Cohen's Factorization Theorem
(see \cite[Thm.~II.5.3.7]{Blackadar2006Theory}).
The quotient of $A$ by this ideal is denoted by $A(Y)$
and the quotient map is denoted by $\pi_Y \colon A \to A(Y)$.
If $Y$ reduces to a point $x$,
we write $C_0(X \setminus x)$ for $C_0(X \setminus \{x\})$,
$A(x)$ for $A(\{x\})$, and $\pi_x$ for $\pi_{\{x\}}$.
One may think of the quotient $A(x)$ as the fiber of $A$ over $x$.
If $a \in A$, we write $a(x)$ for $\pi_x(a)$
and think of $a$ as a function from $X$ to $\coprod_{x \in X} A(x)$.

Suppose that $(A, \eta)$ is a $C(X)$-algebra.
We say that $B$ is a $C(X)$-subalgebra of $(A, \eta)$
if $\eta(C(X))B \subset B$.
Closed ideals of $A$ are always $C(X)$-subalgebras.
Let $F \subset A$, and let $a \in A$.
If $\delta > 0$,
we write $a \in_{\delta} F$ if there exists $b \in F$ such that
$\|a - b\| < \delta$.
The following lemma is well known and sometimes will be used tacitly;
see \cite[Lem.~2.1]{Dadarlat2009continuous}
and \cite[Prop.~C.10]{Williams2007Crossed}.

\begin{lemma}\label{Cast_bundle_lem}
	Let $(A,\eta)$ be a $C(X)$-algebra and
	let $B \subset A$ be a $C(X)$-subalgebra.
	Let $a \in A$ and let $Y \subset X$ be a closed subset.
	\begin{enumerate}
		\item\label{Cast_bundle_lem_semiconti}
			The map $x \mapsto \|a(x)\|$ is upper semicontinuous.
		\item\label{Cast_bundle_lem_norm}
			$\|\pi_Y(a)\| = \max_{x \in Y}\|a(x)\|$.
		\item\label{Cast_bundle_lem_fiberwise_in}
			If $a(x) \in \pi_x(B)$ for all $x \in X$,
			then $a \in B$.
		\item\label{Cast_bundle_lem_fiberwise_indelta}
			If $\delta > 0$ and
			$a(x) \in_{\delta} \pi_x(B)$ for all $x \in X$,
			then $a \in_{\delta} B$.
		\item\label{Cast_bundle_lem_isom}
			The restriction of $\pi_x \colon A \to A(x)$ to $B$
			induces an $*$-isomorphism
			$B(x) \cong \pi_x(B)$ for all $x \in X$.
		\item\label{Cast_bundle_lem_eval}
			If $f \in C(X)$, then $\pi_x(\eta(f)a) = f(x) \pi_x(a)$
			in $A(x)$ for all $x \in X$.
	\end{enumerate} 
\end{lemma}

\begin{definition}
	Let $X$ be a compact Hausdorff space,
	and let $(A,\eta)$ be a unital $C(X)$-algebra.
	Let $G$ be a locally compact group.
	Let $T \colon G \to \mathrm{Homeo}(X)$ be an action on $X$.
	We say that an action $\alpha \colon G \to \Aut(A)$ \emph{lies over} $T$
	if $\eta(f \circ T_{g}^{-1}) = \alpha_{g}(\eta(f))$
	for all $f \in C(X)$ and all $g \in G$.
\end{definition}
Let $u$ denote the canonical unitary in $C^* (\mathbb{Z}, A, \alpha)$.
Then $u \eta(f) u^* = \alpha(\eta(f)) = \eta(f \circ h^{-1})$
for all $f \in C(X)$.

Let $A$ be a unital $C(X)$-algebra.
The fact that $\alpha$ lies over $h$ implies that
for any closed set $Y \subset X$ and any $n \in \mathbb{Z}$,
there is a $*$-isomorphism $\alpha_{n, Y} \colon A(h^{-n}Y) \to A(Y)$
such that \[
\pi_Y(\alpha^n(a)) = \alpha_{n, Y}(\pi_{h^{-n}Y}(a))
\qquad \text{for all} \quad a \in A .
\]
In fact, if $a,b \in A$ such that $a - b \in C_0(X \setminus h^{-n}Y)A$,
then $\alpha(a - b) \in C_0(X \setminus Y) A$.
It is easy to see that the well-defined map
$\alpha_Y \colon A(h^{-1}Y) \to A(Y)$ is a $*$-isomorphism,
bacause $\alpha$ is a $*$-isomorphism of $A$ such that
$\alpha^n [C_0(X \setminus h^{-n}Y)A] = C_0(X \setminus Y)A$.
If $Y$ reduces to a point $x$,
we get a $*$-isomorphism $\alpha_{n, x}$ from $A(h^{-n}x)$ to $A(x)$
such that \[
	\alpha(a)(x) = \alpha_{n, x}(a(h^{-n}x))
	\qquad \text{for all} \quad a \in A
\]
and $\alpha_{n + m, x} = \alpha_{n,x} \circ \alpha_{m, h^{-n}x}$
for all $n,m \in \mathbb{Z}$.

\begin{lemma}\label{bundle_Kishimoto}
	Let $X$ be an infinite compact Hausdorff space,
	and let $(A,\eta)$ be a unital $C(X)$-algebra.
	Suppose that $h: X \to X$ is a minimal homeomorphism.
	If $\alpha \in \Aut(A)$ lies over $h$,
	then $\alpha$ satisfies the Kishimoto's condition.
\end{lemma}

\begin{proof}
	Given $\epsilon > 0$, finite subset $S \subset A$,
	$a \in A_+$ with $\|a\| = 1$, and $n \in \mathbb{Z}_{>0}$,
	we want $r \in A_+$ with $\|r\| = 1$ such that
	$\|rb \alpha^k(r)\| < \epsilon$ for all
	$k \in \{-n , \ldots, n\} \setminus \{0\}$ and all $b \in F$,
	and $\|rar\| > 1 - \epsilon$.
	Since $\|a\| = 1$ and $X$ is compact,
	there exists $x_0 \in X$ such that $\|a(x_0)\| = 1$.
	Since $(X,h)$ is an infinite minimal system,
	there are no periodic points in $X$.
	Choose an open neighborhood $U$ of $x_0$ such that
	the sets $h^{k}(U)$ for $|k| \leq n$ are pairwise disjoint.
	Let $f \in C(X)$ satisfy $0 \leq f \leq 1$,
	$f(x_0) = 1$ and $\supp(f) \subset U$.
	Then $\alpha^k(\eta(f)) = \eta( f \circ h^{-k})$ for $|k| \leq n$
	are pairwise orthogonal in $Z(A)$.
	It is easy to check that $r = \eta(f)$ satisfies the requirements by
	Lemma~\ref{Cast_bundle_lem} \eqref{Cast_bundle_lem_norm}
	and \eqref{Cast_bundle_lem_eval}.
\end{proof}

A $C(X)$-algebra $(A, \eta)$ such that
$x \mapsto \|a(x)\|$ is continuous for all $a \in A$ is
called a continuous $C(X)$-algebra.
A continuous $C(X)$-algebra is the section algebra of a continuous C*-bundle;
see \cite{Lee1976Cast} and \cite{FD1988Representation_I}.
Similar notions appeared in \cite[Def.~4.5.1]{Phillips1987Equivariant}
and \cite[Def.~1.1]{KW1995operations}.
From now on, we omit $\eta$ when the underlying $*$-homomorphism
$C(X) \to Z(M(A))$ is clear and write $fa$ rather than
$\eta(f)a$ for $f \in C(X)$ and $a \in A$.

\begin{lemma}\label{bundle_minimal}
	Let $X$ be an compact Hausdorff space,
	and let $A$ be a continuous unital $C(X)$-algebra such that
	all the fibers are nonzero and simple.
	Suppose that $h: X \to X$ is a minimal homeomorphism.
	If $\alpha \in \Aut(A)$ lies over $h$,
	then $\alpha$ is minimal.
\end{lemma}

\begin{proof}
	Suppose that $B$ is a closed $\alpha$-invariant ideal of $A$.
	We claim that $F = \{x \in X \colon B(x) = 0\}$ is an
	$h$-backward invariant closed set.
	Indeed, if $x \in F$, then $b(x) = 0$ for all $b \in B$.
	Since $B$ is $\alpha$-invariant,
	we have $\alpha(b)(x) = 0$ for all $b \in B$.
	We may assume that $\alpha(b) = f \cdot a$
	for some $f \in C_0(X \setminus x)$ and $a \in A$
	by Cohen's Factorization Theorem.
	We have $b = \alpha^{-1}(f) \alpha^{-1}(a)
	= (f \circ h) \alpha^{-1}(a) \in C_0(X \setminus h^{-1}x)A$,
	from which it follows that $b(h^{-1}x) = 0$ for all $b \in B$.
	So $h^{-1}x \in F$ and $h^{-1}(F) \subset F$.
	Since $A$ is a continuous $C(X)$-algebra, $F$ is closed.
	Minimality of $h$ ensures that $F = \varnothing$ or $F = X$.
	If $F = \varnothing$,
	then $B(x) = A(x)$ by the simplicity of $A(x)$ for all $x \in X$.
	So $B = A$ by Lemma~\ref{Cast_bundle_lem}
	\eqref{Cast_bundle_lem_fiberwise_in}.
	Otherwise $F = X$ and we have $B = 0$
	by Lemma~\ref{Cast_bundle_lem} \eqref{Cast_bundle_lem_norm}.
\end{proof}

\begin{proposition}\label{bundle_simple}
	Let $X$ be an infinite compact Hausdorff space,
	and let $A$ be a continuous unital $C(X)$-algebra such that
	all the fibers are nonzero and simple.
	Suppose that $h \colon X \to X$ is a minimal homeomorphism
	and $\alpha \in \Aut(A)$ lies over $h$.
	Then $C^*(\mathbb{Z}, A, \alpha)$ is simple.
\end{proposition}

\begin{proof}
	 Combine Lemma~\ref{bundle_Kishimoto}, Lemma \ref{bundle_minimal}
	 and Theorem \ref{Kishimoti1981thm}.
\end{proof}

\begin{notation}\label{not_Y_orbit_subbrkalg}
	Let $X$ be an infinite compact Hausdorff space,
	and let $A$ be a continuous unital $C(X)$-algebra such that
	all the fibers are nonzero and simple.
	Suppose that $h \colon X \to X$ is a minimal homeomorphism
	and $\alpha \in \Aut(A)$ lies over $h$.
	Let $u \in C^* (\mathbb{Z}, A, \alpha)$ be the standard unitary.
	Following the notation of Convention \ref{convention_coefficient_ideal},
	for a closed subset $Y \subset X$
	we defined the C*-subalgebra $C^* (\mathbb{Z}, A, \alpha)_{Y}$ to be \[
	C^* (\mathbb{Z}, A, \alpha)_{Y} = C^* (A, C_0(X \setminus Y)Au)
	\subset C^* (\mathbb{Z}, A, \alpha) .
	\]
	We call it the $Y$-\emph{orbit breaking subalgebra}
	of $C^* (\mathbb{Z}, A, \alpha)$.
	If $Y$ reduces to a point $y$, we write $C^* (\mathbb{Z}, A, \alpha)_{y}$
	for $C^* (\mathbb{Z}, A, \alpha)_{\{y\}}$.
\end{notation}

The following proposition is a trivial application of Theorem \ref{main_thm}.
So we omit the proof.
We will replace the condition $f \precsim_{B} a$
with the locally triviality of $A$ and an
additional technical condition on $\alpha$ later.
\begin{proposition}\label{bundle_trivial_appl}
	Adopt Notation~\ref{not_Y_orbit_subbrkalg},
	and let $y \in X$.
	Let $B = C^* (\mathbb{Z}, A, \alpha)_{y}$.
	Suppose that $C^* (\mathbb{Z}, A, \alpha)$ is finite,
	and for any $a \in A_+ \setminus \{0\}$,
	there exists $f \in C(X)_+ \setminus \{0\}$ such that $f \precsim_{B} a$.
	Then $B$ is a large subalgebra of $C^* (\mathbb{Z}, A, \alpha)$
	of crossed product type.
\end{proposition}

Morphisms between $C(X)$-algebras are $C(X)$-linear $*$-homomorphisms.
Let $A$ be a $C(X)$-algebra.
Recall that $A$ is trivial if there exists a C*-algebra $D$ such that
$A$ is isomorphic to $C(X,D)$ as $C(X)$-algebra.
We say that $A$ is locally trivial if for every $x \in X$
there is a compact neighborhood $K$ of $x$
such that $A(K)$ is trivial as a $C(K)$-algebra.
It is obvious that every locally trivial $C(X)$-algebra is continuous.

Let $\phi_x : A(K_x) \to C(K_x , D)$ be a local trivialization near $x \in X$,
and let $\mathrm{int}(K_x)$ denote the interior of $K_x$.
A compactness and shrinking argument gives a finite collection
of local trivializations $\{(K_j , \phi_j)\}$ such that the union of the sets
$\mathrm{int}(K_j)$ covers $X$.
We call such $\{(K_j, \phi_j)\}$ an \emph{atlas} of $A$.
For a closed subset $L \subset K_j$,
denote by $\phi_j |_L$ the $*$-isomorphism from $A(L)$ to $C(L,D)$
induced by $\phi_j$.
In particular, $\phi_j |_x$ gives a $*$-isomorphism from $A(x)$ to $D$.

Since we already know that all fibers at the same $h$-orbit are $*$-isomorphic
when $h$ is minimal and $\alpha$ lies over $h$,
the following lemma is trivial.

\begin{lemma}\label{lem_locally_trivial_and_lie_over}
	Let $X$ be an infinite compact Hausdorff space,
	and let $A$ be a unital $C(X)$-algebra.
	Suppose that $h \colon X \to X$ is a minimal homeomorphism
	and $\alpha \in \Aut(A)$ lies over $h$.
	If $A$ is locally trivial at some $x_0 \in X$,
	then $A$ is locally trivial,
	and all fibers of $A$ are $*$-isomorphic to each other.
\end{lemma}

When $A$ is locally trivial and all fibers of $A$ are $*$-isomorphic to
some C*-algebra $D$, we are able to realize
$\alpha_{n,x} \colon A(h^{-n}x) \to A(x)$ to $\Aut(D)$
(depending on local trivializations) and discuss continuty.
We call $D$ the typical fiber of $A$.
Given two local trivializations $(K, \phi)$ and $(F, \psi)$ of $A$,
for each $n \in \mathbb{Z}$ such that $h^{n}(K) \cap F$ is nonempty
and each $x \in h^{n}(K) \cap F$,
we write \[\alpha_{n,x}^{\phi, \psi}
= (\psi|_{x}) \circ \alpha_{n,x} \circ (\phi|_{h^{-n}x})^{-1} .\]
Then $\alpha_{n,x}^{\phi, \psi} \in \Aut(D)$.
If $\mathcal{A} = \{(K_j, \phi_j) \colon 1 \leq j \leq N\}$ is an atlas of $A$,
we write $\alpha(\mathcal{A})_{n,x}^{i,j}$ for $\alpha_{n,x}^{\phi_i, \phi_j}$.

\begin{lemma}
	Let $X$ be an infinite compact Hausdorff space,
	and let $D$ be a simple unital C*-algebra.
	Let $A$ be a locally trivial $C(X)$-algebra
	with typical fiber $D$.
	Suppose that $h \colon X \to X$ is a minimal homeomorphism
	and $\alpha \in \Aut(A)$ lies over $h$.
	Let $n \in \mathbb{Z}$, and let $x \in X$.
	Suppose that $(K , \phi)$ and $(F , \psi)$
	are local trivializations near $h^{-n}x$ and $x$, respectively.
	Then the map $x \mapsto \alpha_{n,x}^{\phi, \psi}(d)$ from $h^n(K) \cap F$
	to $D$ is continuous for all $d \in D$.
\end{lemma}

\begin{proof}
	Set $L = h^n(K) \cap F$ and consider the following diagram:
	\begin{center}
	\begin{tikzcd}[row sep=normal, column sep=normal]
		A(h^{-n}x) \arrow[rrr, "\alpha_{n,x}"]
			\arrow[ddd, "\phi|_{h^{-n}x}"'] & &
		& A(x) \arrow[ddd, "\psi|_x"] \\
		& A(h^{-n}L) \arrow[r, "\alpha_{n,L}"]
			\arrow[d, "\phi|_{h^{-n}L}"'] \arrow[lu, "\pi_{h^{-n}x}"']
		& A(L) \arrow[ru, "\pi_x"] \arrow[d, "\psi|_L"] & \\
		& {C(h^{-n}L,D)} \arrow[r, "\alpha_{n,L}^{\phi, \psi}"]
			\arrow[ld, "\mathrm{ev}_{h^{-n}x}"']
		& {C(L, D)} \arrow[rd, "\mathrm{ev}_x"] & \\
		D \arrow[rrr, "\alpha_{n,x}^{\phi, \psi}"] &  &  & D \,,                         
	\end{tikzcd}
	\end{center}
	where $\alpha_{n, L}^{\phi, \psi}
	= (\psi|_L) \circ \alpha_{n,L} \circ (\phi|_{h^{-n}L})^{-1}$.
	Since every square is commutative,
	it follows that $\alpha_{n, x}^{\phi, \psi}(d)
	= \alpha_{n,L}^{\phi, \psi} (1 \otimes d) (x)$ for every $d \in D$.
	Hence the map $x \mapsto \alpha_{n,x}^{\phi, \psi}(d)$ is continuous.
\end{proof}

In \cite[Def.~1.8]{ABP2020structure}, Archey, Buck
and the second author introduced the notion of pseudoperiodicity.
Let $D$ be a C*-algebra.
We say that $S \subset \Aut(D)$ is \emph{pseudoperiodic},
if for any $a \in D_+ \setminus \{0\}$,
there exists $b \in D_+ \setminus \{0\}$ such that $b \precsim \gamma(a)$
for all $\gamma \in S \cup \{\mathrm{id}_D\}$.
Now assume that $D$ is a simple unital C*-algebra.
Then every finite subset $S \subset \Aut(D)$ is pseudoperiodic
by \cite[Lem.~2.6]{Phillips2014large};
for the same reason, $S = \{\gamma^n \colon n \in \mathbb{Z}\}$
is pseudoperiodic for every periodic $\gamma \in \Aut(D)$.
It has been proved in \cite[Lem.~1.9, 1.10]{ABP2020structure}
that the set of approximately inner automorphisms of $D$
and compact subsets of $\Aut(D)$ are pseudoperiodic.
When $D$ has strict comparison of positive elements,
or when $D$ has the property (SP)
and the order on projections over $D$ is determined by quasitraces,
the whole group $\Aut(D)$ is pseudoperiodic;
see \cite[Lem.~1.12, 1.14]{ABP2020structure}.

The following lemma follows directly from the definition
and Lemma 2.6 of \cite{Phillips2014large}.
\begin{lemma}\label{lem_pp}
	Let $D$ be a simple unital C*-algebra.
	Let $S_1$ and $S_2$ be pseudoperiodic subsets of $\Aut(D)$.
	Then $S_1 S_2 = \{ \gamma_1 \gamma_2 \colon
	\gamma_1 \in S_1 , \, \gamma_2 \in S_2
	\}$ and $S_1 \cup S_2$
	are pseudoperiodic.
\end{lemma}

\begin{definition}\label{def_ppg}
	Let $X$ be an infinite compact Hausdorff space,
	and let $D$ be a simple unital C*-algebra.
	Let $A$ be a locally trivial $C(X)$-algebra with typical fiber $D$.
	Suppose that $h \colon X \to X$ is a minimal homeomorphism
	and $\alpha \in \Aut(A)$ lies over $h$.
	We say that $\alpha$ is \emph{pseudoperiodic generated} if
	there exists an atlas $\mathcal{A} = \{
		( K_j , \phi_j ) \colon j = 1, \ldots, N \}$
	of $A$ such that the set \[
	\big\{ \alpha(\mathcal{A})_{n,x}^{i,j} \colon
		n \in \mathbb{Z}, \, 
		i, j \in \{1, \ldots, N\}, \,
		x \in h^n(K_i) \cap K_j
	\big\}
	\]
	is pseudoperiodic.
\end{definition}

Let $(K, \phi)$ and $(F, \psi)$ be two local trivializations of $A$.
Define \[\tau_{\phi, \psi}
= (\psi |_{K \cap F}) \circ (\phi |_{K \cap F})^{-1}.\]
Then $\tau_{\phi, \psi}$ gives a $C(K \cap F)$-linear $*$-isomorphism
from $C(K \cap F, D)$ onto itself.
Let $x \in K \cap F$, let $f \in C(K \cap F, D)$,
and define \[t_{\phi, \psi}(x) \colon D \to D \quad
\text{by} \quad f(x) \mapsto \tau_{\phi, \psi}(f)(x).\]
It is easy to check that $t_{\phi, \psi}(x)$ is a well-defined
$*$-automorphism on $D$.
So each $\tau_{\phi, \psi}$ gives a map $t_{\phi, \psi}$ from $K \cap F$ to $\Aut(D)$.
We call $t_{\phi, \psi}$ the \emph{transition map} from $\phi$ to $\psi$.

\begin{lemma}\label{lem_trans_pp}
	Let $X$ be an infinite compact Hausdorff space,
	and let $D$ be a simple unital C*-algebra.
	Let $A$ be a locally trivial $C(X)$-algebra with typical fiber $D$.
	Suppose that $h \colon X \to X$ is a minimal homeomorphism
	and $\alpha \in \Aut(A)$ lies over $h$.
	Suppose that $(K, \phi)$ and $(F, \psi)$ are local trivializations.
	Then the transition map $t_{\phi, \psi} \colon K \cap F \to \Aut(D)$
	from $\phi$ to $\psi$
	induced by $\tau_{\phi, \psi} = (\psi |_{K \cap F}) \circ (\phi |_{K \cap F})^{-1}$
	is continuous in the topology of pointwise convergence in the norm on $D$,
	and	the set $t_{\phi,\psi}(K \cap F)$ is pseudoperiodic.
\end{lemma}

\begin{proof}
	Set $L = K \cap F$.
	Let $x \in L$ and consider the following commutative diagram:
	\begin{center}
	\begin{tikzcd}[row sep=scriptsize, column sep=scriptsize]
	& {C(L, D)} \arrow[rr, "\mathrm{ev}_x"] \arrow[dd, "{\tau_{\phi, \psi}}"]
	& & D \arrow[dd, "{t_{\phi, \psi}(x)}"] \\
	A(L) \arrow[ru, "\phi|_{L}"] \arrow[rd, "\psi|_{L}"']
	& & & \\
	& {C(L, D)} \arrow[rr, "\mathrm{ev}_x"] & & D                           
	\end{tikzcd}
	\end{center}
	Since $t_{\phi, \psi}(x)$ is well-defined,
	we have $t_{\phi, \psi}(x)(d) = \tau_{\phi, \psi}(1 \otimes d)(x)$
	for every $d \in D$.
	Hence $t_{\phi, \psi} \colon L \mapsto \Aut(D)$ is continuous.
	It follows that $t_{\phi, \psi}(L)$ is a compact subset of $\Aut(D)$
	and therefore is pseudoperiodic by Lemma 1.10 of \cite{ABP2020structure}.
\end{proof}

\begin{proposition}
	Adopt Definition~\ref{def_ppg}.
	Suppose that $\alpha \colon A \to A$ is pseudoperiodic generated.
	Then for any atlas
	$\mathcal{B} = \{(F_l, \psi_l) \colon l = 1, \ldots, M\}$ of $A$,
	the set \[
		S_1 = \big\{\alpha(\mathcal{B})_{n,x}^{k,l} \colon
			n \in \mathbb{Z}, \, k, l \in \{1, \ldots, M\}, \,
			x \in h^{n}(F_k) \cap F_l
		\big\}
	\]
	is pseudoperiodic.
\end{proposition}

\begin{proof}
	Let $\mathcal{A} = \{(K_j, \phi_j) \colon j = 1, \ldots, N\}$
	be an atlas of $A$ such that the set \[
	S_0 = \big\{ \alpha(\mathcal{A})_{n,x}^{i,j} \colon
	n \in \mathbb{Z}, \, 
	i, j \in \{1, \ldots, N\}, \,
	x \in h^n(K_i) \cap K_j	\big\}
	\]
	is pseudoperiodic.
	For $i, j \in \{1, \ldots, N\}$ and $k,l \in \{1,\ldots,M\}$,
	let \[
		S_1(k,l,i,j) = \big\{\alpha(\mathcal{B})_{n,x}^{k,l} \colon
		n \in \mathbb{Z}, \,
		x \in h^{n}(F_k \cap K_i) \cap F_l \cap K_j
		\big\} .
	\]
	Then $S_1 = \bigcup_{k,l,i,j} S_1(k,l,i,j)$. By Lemma \ref{lem_pp},
	it suffices to proof that $S_1(k,l,i,j)$ is pseudoperiodic
	for each $(k,l,i,j)$.
	Set $L(n; k,l,i,j) = h^{n}(F_k \cap K_i) \cap F_l \cap K_j$.
	Denote by $t_{\mathcal{A}, \mathcal{B}}^{j,l} \colon
	K_{j} \cap F_{l} \to \Aut(D)$
	the transition map from $\phi_{j}$ to $\psi_{l}$.
	For $x \in L(n; k,l,i,j)$, we claim that \[
		\alpha(\mathcal{B})_{n,x}^{k,l}
		= \big( t_{\mathcal{A}, \mathcal{B}}^{j, l}(x) \big)
		\circ \alpha(\mathcal{A})_{n, x}^{i, j}
		\circ \big( t_{\mathcal{A}, \mathcal{B}}^{i, k}(h^{-n}x) \big)^{-1}.
	\]
	In fact, abbreviating $L(n; k,l,i,j)$ to $L$,
	we have the following commutative diagram: \begin{center}
	\begin{tikzcd}[row sep=normal, column sep=scriptsize]
		D \arrow[dddd, "{t_{\mathcal{A},\mathcal{B}}^{i,k}(h^{-n}x)}"']
		\arrow[rrrrr, "{\alpha(\mathcal{A})_{n,x}^{i,j}}"] & & & & &
		D
		\arrow[dddd, "{t_{\mathcal{A},\mathcal{B}}^{j,l}(x)}"'] \\
		& {C(h^{-n}L,D)} \arrow[lu, "\mathrm{ev}_{h^{-n}x}"']
		\arrow[rrr, "{\alpha(\mathcal{A})_{n,L}^{i,j}}"]
		\arrow[dd, "{\tau_{\mathcal{A},\mathcal{B}}^{i,k}|_{h^{-n}L}}"'] & & &
		{C(L,D)}
		\arrow[ru, "\mathrm{ev}_{x}"]
		\arrow[dd, "{\tau_{\mathcal{A},\mathcal{B}}^{j,l}|_L}"'] & \\
		& & A(h^{-n}L) \arrow[r, "{\alpha_{n,L}}"]
		\arrow[lu, "\phi_i|_{h^{-n}L}"']
		\arrow[ld, "\psi_k|_{h^{-n}L}"] & A(L)
		\arrow[ru, "\phi_j |_L"]
		\arrow[rd, "\psi_l|_{L}"'] & &  \\
		& {C(h^{-n}L,D)}
		\arrow[ld, "\mathrm{ev}_{h^{-n}x}"]
		\arrow[rrr, "{\alpha(\mathcal{B})_{n,L}^{k,l}}"'] & & &
		{C(L,D)} \arrow[rd, "\mathrm{ev}_{x}"] & \\
		D \arrow[rrrrr, "{\alpha(\mathcal{B})_{n,x}^{k,l}}"] & & & & & D,
	\end{tikzcd}
	\end{center}
	where \begin{gather*}
	\tau_{\mathcal{A},\mathcal{B}}^{j,l}
	= (\psi_l |_{K_j \cap F_l}) \circ (\phi_j |_{K_j \cap F_l})^{-1} , \\
	\alpha(\mathcal{A})_{n,L}^{i,j}
	= (\phi_j|_{L}) \circ \alpha_{n,L} \circ (\phi_i|_{h^{-n}L})^{-1}, \\
	\text{and} \enspace
	\alpha(\mathcal{B})_{n,L}^{k,l}
	= (\psi_l|_{L}) \circ \alpha_{n,L} \circ (\psi_k|_{h^{-n}L})^{-1}.
	\end{gather*}
	Write \[
	S_0(i,j,k,l) = \big\{ \alpha(\mathcal{A})_{n,x}^{i,j} \colon
		n \in \mathbb{Z}, \, x \in h^n(F_k \cap K_i) \cap F_l \cap K_j
	\big\}.
	\]
	Then $S_0(i,j,k,l)$ is pseudoperiodic and \[
		S_1(k,l,i,j) \subset
			t_{\mathcal{A},\mathcal{B}}^{i,k}(L)
			S_0(i,j,k,l)
			t_{\mathcal{A},\mathcal{B}}^{j,l}(L)
	\]
	by our claim.
	Thus $S_1(k,l,i,j)$ is pseudoperiodic by Lemma~\ref{lem_trans_pp} and \ref{lem_pp}.
\end{proof}

Let $K, F \subset X$ be closed subsets such that
$\mathrm{int}(K) \cup \mathrm{int}(F) = X$.
Generally we cannot conclude that $a \precsim_A b$ from
$\pi_K(a) \precsim_{A(K)} \pi_K(b)$ and
$\pi_F (a) \precsim_{A(F)} \pi_F (b)$ without other conditions,
but the following lemma is trivial.
\begin{lemma}\label{bundle_Cuntz}
	Let $X$ be a compact Hausdorff space,
	and let $A$ be a unital $C(X)$-algebra with $a,b \in A_+$.
	If there exists a compact subset $K \subset X$ such that
	$\pi_K(a) \precsim_{A(K)} \pi_K(b)$	and $a(x) = b(x) = 0$
	for all $x \in X \setminus K$, then $a \precsim_A b$.
\end{lemma}

\begin{proof}
	Let $\epsilon > 0$.
	Then there exists $v' \in A(K)$ such that \[
	\| \pi_K(a) - v' \pi_K(b) (v')^*\| < \epsilon .\]
	Let $v \in A$ be a lift of $v' \in A(K)$.
	Then $\| a(x) - v(x) b(x) v(x)^*\| < \epsilon$ for all $x \in K$
	and all $x \in X \setminus K$.
	Thus $\| a - vbv^*\| < \epsilon$ by Lemma~\ref{Cast_bundle_lem}
	\eqref{Cast_bundle_lem_norm}.
\end{proof}

\begin{lemma}\label{directed_subalg}
	Let $X$ be an infinite compact Hausdorff space, 
	and let $D$ be a simple unital C*-algebra.
	Suppose that $A$ is a locally trivial $C(X)$-algebra
	with typical fiber $D$.
	Let $h \colon X \to X$ be a minimal homeomorphism
	and let $\alpha \in \Aut(A)$ lie over $h$.
	Fix a point $y \in X$.
	Let $u$ denote the canonical unitary in $C^* (\mathbb{Z}, A, \alpha)$,
	and let $B$ be the C*-subalgebra generated by $A$ and $C_0(X \setminus y) A u$
	in $C^*(\mathbb{Z}, A, \alpha)$.
	Suppose that $\alpha$ is pseudoperiodic generated.
	Then for any $a \in A_+ \setminus \{0\}$,
	there exists $f \in C(X)_+ \setminus \{0\}$
	such that $f \precsim_{B} a$.
\end{lemma}

\begin{proof}
	Let $\mathcal{A} = \{(K_j, \phi_j) \colon j = 1, \ldots, N\}$ be
	an atlas of $A$ such that the corresponding set \[
		S = \big\{ \alpha(\mathcal{A})_{m,x}^{i,j} \colon m \in \mathbb{Z}, \,
			i, j \in \{1, \ldots, N\}, \,
			x \in h^m(K_i) \cap K_j
		\big\}
	\]
	is pseudoperiodic.
	Let $\pi_j \colon A \to A(K_j)$ be the quotient map.
	Choose an open set $U \subset X$
	such that $\|a(x)\| > 0$ for all $x \in U$.
	Take $x_1 \in U \setminus \{y\}$ and suppose without loss of generality
	that $x_1 \in \mathrm{int}(K_1)$.
	Take an open neighborhood $U_1$ of $x_1$
	such that $\overline{U_1} \subset \mathrm{int}(K_1) \cap U$
	and $y \notin \overline{U_1}$.
	Choose $f_1 \in C(X)$ such that $0 \leq f_1 \leq 1$, $f_1(x_1) = 1$
	and $\supp(f_1) \subset U_1$. Set $a_1 = f_1 a$.
	Then $a_1 \precsim_{A} a$. Use Kirchberg's Slice Lemma
	(see \cite[Lem.~4.1.9]{RS2002Classification})
	to find $g_0 \in C_0(U_1)_+ \setminus \{0\}$ and
	$d_0 \in D_+ \setminus \{0\}$ such that \[
		g_0 \otimes d_0 \precsim_{C_0(U_1,D)} \phi_1 (\pi_{1} ( a_1 )) .
	\]
	Since $S$ is pseudoperiodic,
	there exists $d_1 \in D_+ \setminus \{0\}$ such that \[
	\gamma(d_1) \precsim_D d_0 \qquad \text{for all} \quad \gamma \in S.
	\]
	Without loss of generality,
	we assume that $\|d_1\| = 1$ and set $d = (d_1 - \frac{1}{2})_+$.
	Choose $y_0 \in U_1 \cap \{x \colon g_0(x) \ne 0\}
	\setminus \{h^n(y) : n \in \mathbb{Z}\}$ by the Baire Category Theorem.
	Choose an open neighborhood $U_2$ of $y_0$ such that
	$\overline{U_2} \subset U_1$ and such that
	$g_0$ has positive minimum value $\rho$ on $\overline{U_2}$.
	
	Choose $f_2 \in C(X)$ such that $0 \leq f_2 \leq 1$,
	$f_2 |_{\overline{U_2}} = 1$ and $\supp(f_2) \subset U_1$.
	Let $b \in A_+$ be a lifting of $\phi_1^{-1}(f_2 \otimes d) \in A(K_1)_+$
	such that $b(x) = 0$ for $x \in X \setminus \mathrm{int}(K_1)$.
	Then $(\phi_1|_x) (b(x)) = f_2(x)d$ in $D$ for all $x \in K_1$.
	Since $A(y_0) \cong D$ is a simple unital C*-algebra,
	there exists $k > 0$ and
	$z_1^{\prime} , z_2^{\prime} , \ldots , z_k^{\prime} \in A(y_0)$
	such that \[
		\sum_{j=1}^k z^{\prime}_j b(y_0) z^{\prime *}_j = 1
		\quad \text{in} \quad A(y_0).
	\]
	Lifting $z^{\prime}_j \in A(y_0)$ to $z_j \in A$ for each $j$,
	we can find an open neighborhood
	$U_3$ of $y_0$ such that $\overline{U_3} \subset U_2$ and \[
		\left\| 1 - \sum_{j=1}^k z_j(x) b(x) (z_j)^* (x) \right\|
		< \frac{1}{4}
		\qquad \text{for all} \quad x \in \overline{U_3} .
	\]
	Set \[
		w^{\prime}_j =
		\bigg( \sum_{i = 1}^k \pi_{\overline{U_3}}(z_j b z^*_j) \bigg)^{-1/2}
		\pi_{\overline{U_3}}(z_j)
		\in A(\overline{U_3}) .
	\]
	Then $\sum_{j=1}^k w'_j \pi_{\overline{U_3}}(b) w'_j = 1$
	in $A(\overline{U_3})$.
	Let $w_j \in A$ be a lift of $w^{\prime}_j \in A(\overline{U_3})$
	for $j = 1,2,\cdots,k$.
	Then we have \[
		\sum_{j=1}^k w_j(x) b(x) w_j^*(x) = 1 \qquad \text{in} \quad A(x)
		\quad \text{for} \quad x \in \overline{U_3}.
	\]
	Since $(X,h)$ is an infinite minimal system,
	the forward orbits of $y_0$ will return to $U_3$ infinite many times.
	So there exists $n_1 , n_2 , \ldots , n_k$ 
	with $0 = n_1 < n_2 < \cdots < n_k$ 
	such that $h^{n_j} (y_0) \in U_3$ for $j = 1, 2, \ldots, k$.
	
	To simplify the notation,
	we write $\beta_{n,x}^{i,j} = \alpha(\mathcal{A})_{n,x}^{i,j}$.
	Choose an open neighborhood $V$ of $y_0$ which is so small
	such that the following hold: \begin{enumerate}
		\item $\overline{V} \subset U_3 
			\cap \bigcap_{m=0}^{n_k}
			\left\{h^{-m}(\mathrm{int}(K_i))
				\mid h^m(y_0) \in \mathrm{int}(K_i)
			\right\}
			\setminus \bigcup_{m=0}^{n_k} \{h^{-m}(y)\}$.
		\item $h^{n_j}(\overline{V}) \subset U_3$ for $j = 1,2,\cdots,k$.
		\item The sets $\overline{V} , h(\overline{V}) ,
			\ldots , h^{n_k}(\overline{V})$ are disjoint.
		\item For any $z \in V$, any $m \in \{0,1,\cdots, n_k\}$, 
		and any $i$ such that $h^m(y_0) \in \mathrm{int}(K_i)$,
		we have \[
			\left\| \beta_{m,h^m z}^{1,i}(d_1)
			- \beta_{m,h^m y_0}^{1,i}(d_1) \right\| < \frac{1}{4}
			\qquad \text{if} \quad z \in K_1 \cap h^{-m} (K_{i}) .
		\]
	\end{enumerate}
	Choose $f, f_0 \in C(X)_+$ such that $\|f\| , \|f_0\| \leq 1$,
	W$f(y_0) = 1$, $f f_0 = f$ and $\supp(f_0) \subset V$.
	For $m = 0 , 1 , \cdots , n_k$,
	set $c_m = \alpha^m(fb) = (f \circ h^{-m}) \alpha^m(b) \in A_+$.
	Then \[
		c_m(x) = f(h^{-m}x) \alpha_{m,x}(b(h^{-m}x))
	\]
	for all $x \in X$ and \[
		c_m (x) = \begin{cases}
			f(h^{-m}x) (\phi_{i}|_x)^{-1}[\beta_{m,x}^{1,i} (d)] \quad
			& x \in h^{m}(K_1) \cap K_i \\
			0 \quad
			& x \in X \setminus h^m(V) .
		\end{cases}
	\]
	Let $m = 0, 1, \ldots, n_k$.
	Fix $i = i(m)$ such that $h^m(y_0) \in \mathrm{int}(K_i)$.
	We can assume that $i(n_j) = 1$ for all $j = 1, 2, \cdots , k$
	by condition (2).
	Set \[
		d_2 = \beta_{m, h^m y_0}^{1,i(m)} (d_1) \in D_+
		\qquad \text{and} \qquad L = h^m(\supp(f_0)) .
	\]
	We have $L \subset h^m(\mathrm{int}(K_1)) \cap \mathrm{int}(K_i)$
	by condition (1). Let \[
		c^{\prime}_m (x) = \begin{cases}
			(\phi_i |_x)^{-1} f(h^{-m}x) d_2 \quad & x \in L\\
			0 \quad & x \in X \setminus L .
		\end{cases}
	\]
	We claim that \begin{equation}\label{bundle_piece_1}
		c_m \precsim_{A} c^{\prime}_m \qquad
		\text{for all} \quad m = 0,1,\ldots, n_k .
	\end{equation}
	By Lemma~\ref{bundle_Cuntz},
	it suffices to show that $\pi_L(c_m) \precsim_{A(L)} \pi_L(c_m^\prime)$ .
	To see this, let $\epsilon > 0$.
	Define $b^{\prime} \in A(L)$ by
	$b^{\prime}(x) = (\phi_i |_x)^{-1} [\beta_{m,x}^{1,i}(d_1)]$ for $x \in L$.
	Then \begin{align*}
		\big(b'(x) - \tfrac{1}{2} \big)_+
		& = \big(\alpha_{m,x} [(\phi_1|_{h^{-m}x})^{-1}(d)] - \tfrac{1}{2} \big)_+ \\
		& = \alpha_{m,x} \Big[(\phi_1|_{h^{-m}x})^{-1}
			\big[ (d - \tfrac{1}{2})_+ \big] \Big]
		= \alpha_{m,x}[b(h^{-m}x)]
	\end{align*}
	for all $x \in L \subset h^m(V) \subset h^m(U_2)$.
	Then condition (4) implies that \begin{align*}
	\frac{1}{4} &
	\geq \big\| (\phi_i|_x)^{-1} [\beta_{m,x}^{1,i}(d_1)]
		- (\phi_i|_x)^{-1} [\beta_{m,h^m y_0}^{1,i}(d_1)] \big\| \\
	& = \big\| b'(x) - (\phi_i|_x)^{-1} (d_2) \big\| \\
	& = \big\| b'(x) - (\phi_{i}|_L)^{-1}(1 \otimes d_2)(x) \big\|
	\end{align*}
	for all $x \in L$.
	Thus $\| b' - (\phi_{i}|_L)^{-1}(1 \otimes d_2)\|_{A(L)} < \frac{1}{2}$.
	So \[
		\big(b' - \tfrac{1}{2}\big)_+
		\precsim_{A(L)} (\phi_{i}|_L)^{-1}(1 \otimes d_2)
	\]
	by Lemma~\ref{cuntz_comp_lem} \eqref{cuntz_comp_lem_4}.
	Therefore there exists $v_0 \in A(L)$ such that \[
		\Big\| \big(b' - \tfrac{1}{2}\big)_+ -
		v_0 (\phi_{i}|_L)^{-1}(1 \otimes d_2) v_0^* \Big\|
		< \frac{\epsilon}{2} . \]
	Then for any $x \in L$, we have 
	\begin{align*}
	& \left\| v_0(x) c_m^\prime(x) v_0^*(x) - c_m(x) \right\| \\
	& \hspace*{3em} \mbox{} = \left\|
		f(h^{-m}x) v_0(x) (\phi_i|_L)^{-1}(1 \otimes d_2)(x) v_0^*(x)
		- f(h^{-m}x) \alpha_{m,x}(b(h^{-m}x)) \right\| \\
	& \hspace*{3em} \mbox{} = f(h^{-m}x) \left\|
		v_0(x) (\phi_i|_L)^{-1}(1 \otimes d_2)(x) v_0^*(x)
		- \alpha_{m,x}(b(h^{-m}x)) \right\| \\
	& \hspace*{3em} \mbox{} = f(h^{-m}x) \Big\|
		v_0(x) (\phi_{i}|_L)^{-1}(1 \otimes d_2)(x) v_0^*(x)
		- \big(b'(x) - \tfrac{1}{2}\big)_+ \Big\| \\
	& \hspace*{3em} \mbox{} \leq \Big\|
		v_0 (\phi_{i}|_L)^{-1}(1 \otimes d_2) v_0^*
		- \big(b' - \tfrac{1}{2}\big)_+ \Big\|
	< \frac{\epsilon}{2} .
	\end{align*}
	Taking the supremum over $x \in L$ gives
	$\|v_0 \pi_L(c_m^\prime) v_0^* - \pi_L(c_m)\| < \epsilon$.
	Thus the claim follows.
	
	Next we claim that \[
		f \precsim_{B} \sum_{j=1}^k c_{n_j} .
	\]
	Recall that we have found $w_j \in A$ for $j = 1,2,\ldots,k$
	such that \[
		\sum_{j=1}^k w_j(x) b(x) w_j^*(x) = 1 \qquad \text{in} \quad A(x)
		\enspace \text{for} \enspace x \in \overline{U_3}.
	\]
	Define $v_j = (f_0 \circ h^{-n_j}) u^{n_j} w_j^*
	\in C^*(\mathbb{Z},A,\alpha)$.
	Then \[
	E(v_j u^{-n}) = \begin{cases}
		(f_0 \circ h^{-n_j}) \alpha^{n_j}(w_j^*) \quad & n = n_j \\
		0 \quad & n \ne n_j .
	\end{cases}
	\]
	Since $E(v_j u^{-n})$ vanishes on $\{y , h(y) , \ldots , h^n(y)\}$
	by condition (1),
	we conclude that $v_j \in B$ by Proposition~\ref{piece_1}.
	We calculate: \begin{align*}
		v_j^* c_{n_j} v_j &= w_j u^{-n_j} (f \circ h^{-n_j})
			\alpha^{n_j}(fb) (f \circ h^{-n_j}) u^{n_j} w_j^* \\
		& = w_j f_0 (fb) f_0 w_j^* = f(w_j b w_j^*) .
	\end{align*}
	Using $\supp(f) \subset U_3$ and condition \eqref{cuntz_comp_lem_2}
	at the first step, and Lemma~\ref{cuntz_comp_lem} \eqref{cuntz_comp_lem_5}
	at the second step,	and using condition \eqref{cuntz_comp_lem_3} and
	Lemma~\ref{cuntz_comp_lem} \eqref{cuntz_comp_lem_6} at the last step,
	we have \begin{equation}\label{bundle_piece_2}
		f = \sum_{j=1}^k f (w_j b w_j^*) 
		\precsim_{B} \bigoplus_{j=1}^{k} c_{n_j}
		\sim_{B} \sum_{j=1}^{k} c_{n_j} .
	\end{equation}
	This proves our second claim.
	
	Finally we claim that $\sum_{j=1}^k c_{n_j}^{\prime} \precsim_{A} a_1$.
	It suffices to prove that
	\[
		\sum_{j=1}^k \pi_1 (c^\prime_{n_j})
		\precsim_{A(K_1)} \phi_1^{-1} (g_0 \otimes d_0)
	\]
	by Lemma~\ref{bundle_Cuntz}.
	In fact, using \cite[Lem.~1.11]{Phillips2014large} at step 2,
	we have \begin{align*}
		\sum_{j=1}^k \phi_1 \left(\pi_{1} ( c^{\prime}_{n_j} )\right)
		& = \sum_{j=1}^k (f \circ h^{-n_j}) \otimes
			\beta_{n_j,h^{n_j} y_0}^{1,1}(d_1) \\
		& \precsim_{C(K_1, D)} \bigoplus_{j=1}^k
			(f \circ h^{-n_j}) \otimes d_0 \\
		& \sim_{C(K_1, D)} \sum_{j=1}^k (f \circ h^{-n_j}) \otimes d_0 \\
		& = \Big( \sum_{j=1}^k f \circ h^{-n_j} \Big) \otimes d_0 \\
		& \precsim_{C(K_1, D)} \frac{g_0}{\rho} \otimes d_0
		\sim_{C(K_1, D)} g_0 \otimes d_0 .
	\end{align*}
	Thus we get \begin{equation}\label{bundle_piece_3}
	\sum_{j=1}^k c'_{n_j} \precsim_{A} a_1 .\end{equation}

	Now combine \eqref{bundle_piece_1}, \eqref{bundle_piece_2}
	and \eqref{bundle_piece_3} to get $f \precsim_{B} a_1 \precsim_{B} a$.
\end{proof}

\begin{lemma}\label{lem_finiteness}
	Let $X$ be an infinite compact Hausdorff space, 
	and let $D$ is a simple unital C*-algebra 
	with nonempty tracial state space.
	Let $A$ be a locally trivial $C(X)$-algebra with typical fiber $D$,
	Suppose that $h \colon X \to X$ is a minimal homeomorphism
	and $\alpha \in \Aut(A)$ lies over $h$.
	Then $C^*(\mathbb{Z}, A, \alpha)$ has a tracial state and is stably finite.
\end{lemma}

\begin{proof}
	Since we already know that $A$ is simple
	by Proposition~\ref{bundle_simple},
	it remains to show that $C^*(\mathbb{Z}, A, \alpha)$ has a tricial state.
	Take $x \in X$ and let $\phi$ be a local trivialization near $x$.
	Let $\tau$ be a tracial state on $D$.
	Then $\tau \circ (\phi|_x) \circ \pi_x$ gives a tracial state on $A$.
	Since $A$ is unital, the tracial state space $\mathrm{T}(A)$
	is a nonempty compact convex subset,
	and $\tau \mapsto \tau \circ \alpha$ gives a
	continuous affine transformation on it.
	Applying the Markov-Kakutani fixed point theorem,
	we get an $\alpha$-invariant tracial state $\tau_{\alpha}$ on $A$.
	Let $E \colon C^* (\mathbb{Z}, A, \alpha) \to A$
	be the standard condition expectation.
	Then $\tau_{\alpha} \circ E$ gives a tracial state
	on $C^*(\mathbb{Z}, A, \alpha)$;
	see for example \cite[Example~10.1.30]{GKPT2018Crossed}.
\end{proof}

We then give the following theorem as promised.

\begin{theorem}\label{thm_for_application}
	Let $X$ be an infinite compact Hausdorff space, 
	and let $D$ be a simple unital C*-algebra
	with nonempty tracial state space.
	Let $A$ be a locally trivial $C(X)$-algebra with typical fiber $D$.
	Suppose that $h \colon X \to X$ is a minimal homeomorphism
	and $\alpha \in \Aut(A)$ lies over $h$.
	Let $y \in X$ be arbitary.
	If $\alpha$ is pseudoperiodic generated,
	then $C^*(\mathbb{Z}, A, \alpha)$ has a large subalgebra of crossed type
	associated to the closed ideal $C_0(X \setminus y)A$.
\end{theorem}

\begin{proof}
	Combine Proposition~\ref{bundle_trivial_appl},
	Lemma~\ref{directed_subalg} and Lemma~\ref{lem_finiteness}.
\end{proof}

\section{Mapping Torus and the recursive subhomogeneous structure}\label{mapping_torus_sec}

Given a C*-algebra $D$ and a $*$-automorphism $\gamma$ of $D$,
the mapping torus of the pair $(D, \gamma)$ is defined by \[
	M_{\gamma} = \{a \in C([0,1], D) \colon a(1) = \gamma(a(0))\}.
\]
The mapping torus can also be defined as the pullback in the diagram \begin{center}
	\begin{tikzcd}
		& D \arrow[d, "{(\mathrm{id}, \gamma)}"] \\
		{C([0,1], D)} \arrow[r, "{\rho_{\{0,1\}}}"] & D\oplus D,                             
	\end{tikzcd}
\end{center}
where $\rho_{\{0,1\}}$ is the restriction map.
The mapping torus has appeared in 
\cite{Paschke1983mapping}, \cite[Def.~5.2.14]{Lin2001Intro},
and \cite[Sect.~4.5]{APS2011pullbacks}.

Let $S^1 = \{z \in \mathbb{C} \colon |z| = 1\}$.
The multiplier algebra $M(M_{\gamma})$ is the set of
strictly continuous functions $a \colon [0,1] \to M(D)$ such that
$a(1) = M(\gamma)(a(0))$.
Then the $*$-homomorphism \[
	\eta \colon C(S^1) \to Z(M(M_{\gamma})) , \quad 
	\eta(f)(t) = f(\mathrm{e}^{2\pi \mathrm{i} t}) 1_{M(D)}
	\enspace \text{for} \enspace t \in [0,1]
\]
makes $M_{\gamma}$ a $C(S^1)$-algebra.
Nevertheless, we only focus on the unital case.

\begin{remark}
	We can identify $C(S^1)$ with
	$\big\{ f \in C([0,1]) \colon f(0) = f(1) \big\}$
	via the map sending $g \in C(S^1)$ to the function
	$f(t) = f(\mathrm{e}^{2\pi \mathrm{i}t})$ for $t \in [0,1]$.
	This gives a shorter description of the structure map such that
	$[\eta(f)a](t) = f(t)a(t)$ for all $a \in M_{\gamma}$.
	We switch as convenient between the two $S^1$ formulations.
\end{remark}

\begin{lemma}\label{mapping_torus_locally_trivial}
	Let $D$ be a unital C*-algebra,
	and let $\gamma \in \Aut(D)$.
	The mapping torus $M_{\gamma}$ is a locally trivial $C(S^1)$-algebra.
\end{lemma}

\begin{proof}
	Let $z \in S^1$. If $z \ne 1$,
	then we can take closed neighborhood $K$ of $z$ such that
	$K \cap \{1\} = \varnothing$.
	The restriction on $K$ gives a $C(K)$-linear $*$-isomorphism
	from $M_{\gamma}(K)$ to $C(K, D)$.
	If $z = 1$, we take $K = \{z \in S^1 \colon \mathrm{Re}(z) \geq 0\}$.
	Then under the identification \[
	M_{\gamma}(K) = \Big\{
	a \in C\big(\big[0,\tfrac{1}{4}\big] \cup \big[\tfrac{3}{4}, 1\big], D\big)
	\colon a(1) = \gamma(a(0)) \Big\} .
	\]
	Now \[
	C = \Big\{	a \in
	C\big(\big[0,\tfrac{1}{4}\big] \cup \big[\tfrac{3}{4}, 1\big], D\big)
	\colon a(1) = a(0) \Big\} 
	\]
	is clearly a trivial $C(K)$-algebra,
	and the map $\varphi \colon M_{\gamma}(K) \to C$ defined by \[
		\varphi(a)(t) = \begin{cases}
			a(t)  \quad & t \in \left[0, \frac{1}{4}\right] \\
			\gamma^{-1}[a(t)] \quad  & t \in \left[\frac{3}{4}, 1\right]
		\end{cases}
	\]
	gives an isomorphism of $C(K)$-algebras.
\end{proof}

The following lemma is trivial by Lemma~\ref{mapping_torus_locally_trivial}.
For convenience, we give all the maps explicitly.

\begin{lemma}\label{lem_pullback_mapping_torus}
	Let $K_1$, $K_2$ be closed subsets of $S^1$ such that
	under the identification \[
	K_1 = \big[0, \tfrac{1}{4}\big] \cup \big[ \tfrac{3}{4}, 1\big]\quad \text{and}
	\quad K_2 = \big[ \tfrac{1}{8}, \tfrac{7}{8} \big] .
	\]
	The mapping torus $M_{\gamma}$ in Lemma \ref{mapping_torus_locally_trivial}
	is isomorphic to the pullback of the following diagram:
	\begin{center}
		\begin{tikzcd}
		& {C(K_1, D)} \arrow[d, "\psi_1"] \\
		{C(K_2, D)} \arrow[r, "\psi_2"] & {C(K_1 \cap K_2, D),}           
		\end{tikzcd}
	\end{center}
	where \[
		\psi_1(a)(t) = \begin{cases}
			a(t) \, & t \in \left[ \tfrac{1}{8}, \tfrac{1}{4} \right] \\
			\gamma(a(t)) \, & t \in \left[ \tfrac{3}{4}, \tfrac{7}{8} \right]
		\end{cases}
		\enspace \quad \text{and} \quad
		\psi_2(a) = a |_{K_1 \cap K_2} .
	\]
	In fact, the map \[
		M_{\gamma} \to C(K_1, D) \oplus_{\psi_0, \psi_1} C(K_2, D),
		\quad a \mapsto \big( \phi_0(a), \phi_1(a) \big),
	\]
	where \[
		\phi_1(a)(t) = \begin{cases}
			a(t) \, & t \in \left[ 0, \tfrac{1}{4} \right] \\
			\gamma^{-1}(a(t)) \, & t \in \left[ \tfrac{3}{4}, 1 \right]
		\end{cases}
		\enspace \quad \text{and} \quad
		\phi_2(a) = a |_{K_2}
	\]
	gives an isomorphism,
	and the set $\mathcal{A} = \big\{(K_1, \phi_1), (K_2, \phi_2)\big\}$
	is an atlas of $M_{\gamma}$.
\end{lemma}

\begin{lemma}\label{lem_homotopy_and_trivial_bundle}
	Let $D$ be a unital C*-algebra and let $\gamma \in \Aut(D)$.
	Then $M_{\gamma}$ is a trivial $C(S^1)$-algebra
	if and only if $\gamma$ is homotopic to $\mathrm{id}_D$ in $\Aut(D)$.
\end{lemma}

\begin{proof}
	Identify $C(S^1, D)$ with \[
		\big\{ a \in C([0,1], D) \colon a(0) = a(1)\big\} .
	\]
	If there exists $*$-automorphisms $\psi_t \colon D \to D$
	for $t \in [0,1]$ such that $\psi_0 = \mathrm{id}_D$, $\psi_1 = \gamma$
	and for every fixed $d \in D$ the map $[0,1] \to D$,
	$t \mapsto \psi_t(d)$, is continuous,
	then the map $\varphi: C(S^1, D) \to M_{\gamma}$ defined by \[
		\varphi(b)(t) = \psi_t(b(t))
	\]
	gives an isomorphism of $C(S^1)$-algebras.
	
	Now assume that $\varphi \colon C(S^1, D) \to M_{\gamma}$
	is a $C(S^1)$-linear $*$-isomorphism.
	Define a family of $*$-homomorphisms $\sigma_t \colon D \to D$ by
	$\sigma_t(d) = \varphi(1 \otimes d)(t)$ for $t \in [0,1]$.
	Then $\sigma_0 = \mathrm{id}_D$, $\sigma_1 = \gamma$
	and $t \mapsto \sigma_t(d)$ is continuous for all $d \in D$.
	It remains to show that $\sigma_t$, for each $t$, is in $\Aut(D)$.
	
	Let $d \in D \setminus \{0\}$. 
	Since the function $1 \otimes d$ is nowhere vanishing
	and $\varphi$ is $C(S^1)$-linear,
	$\varphi(1 \otimes d)$ is also nowhere vanishing.
	So $\sigma_t$ is injective.
	On the other hand,
	let $d^\prime \in D$ and $t_0 \in [0,1]$ be arbitrary.
	There exists $b \in M_{\gamma}$ such that $b(t_0) = d^\prime$.
	Set $d_1 = \varphi^{-1}(b)(t_0)$.
	Then $\varphi^{-1}(b) - 1 \otimes d_1$ vanishes at $t_0$,
	and hence $b - \varphi(1 \otimes d_1)$ vanishes at $t_0$.
	So $\sigma_{t_0}(d_1) = d^\prime$.
	This proves that $\sigma_t$ is surjective and completes the proof.
\end{proof}

Let $\zeta$ be an irrational number with $0 < \zeta < 1$,
and let $h(z) = \mathrm{e}^{2\pi\mathrm{i}\zeta}z$ for $z \in S^1$.
Then $h$ sends $[0, 1 - \zeta]$ to $[\zeta, 1]$ and
sends $[1 - \zeta, 1]$ to $[0, \zeta]$ under the idenification.
The following lemma is an abstraction of this example,
and is well known in dynamical systems involving rotation numbers.

\begin{lemma}\label{orientation_piecewise}
	Let $h \colon S^1 \to S^1$ be an orientation perserving homeomorphism.
	Then there are unique $t_0 \in (0,1]$, $s_0 \in [0,1)$
	and strictly increasing continuous functions
	$k_0 \colon [0, t_0] \to [s_0, 1]$ and $k_1 \colon [t_0, 1] \to [0, s_0]$
	such that $k_0(t_0) = 1$, $k_1(t_0) = 0$, $k_0(0) = k_1(1) = s_0$, and \[
		h(\mathrm{e}^{2 \pi \mathrm{i} t}) = \begin{cases}
			\mathrm{e}^{2 \pi \mathrm{i} k_0(t)} \quad & t \in [0, t_0] \\
			\mathrm{e}^{2 \pi \mathrm{i} k_1(t)} \quad & t \in [t_0, 1].
		\end{cases}
	\]
\end{lemma}

\begin{proof}
	Let $k \colon \mathbb{R} \to \mathbb{R}$ be a lift of $h:S^1 \to S^1$.
	Then $k$ is strictly increasing and satisfies $k(t+1) = k(t) + 1$
	and $\exp(2 \pi \mathrm{i} k(t)) = h(\mathrm{e}^{2 \pi \mathrm{i} t})$
	for all $t \in \mathbb{R}$.
	Suppose that $k(0) \in [n, n+1)$ for some $n \in \mathbb{Z}$.
	Let $t_0 = \min(\{\tau \in (0, 1] \colon k(\tau) \in \mathbb{Z}\})$,
	$s_0 = k(0) - n$,
	$k_0 (t) = k(t) - n$,
	and $k_1(t) = k(t) - k(t_0)$.
	All the equations are easy to check.
	Note that $t_0$ is unique as a lift of $h^{-1}(1)$ in $(0,1]$
	and $s_0$ is unique as a lift of $h(1)$ in $[0,1)$.
	The rest is trivial.
\end{proof}

\begin{proposition}\label{action_on_ZDgamma}
	Let $D$ be a unital C*-algebra, and let $\gamma \in \Aut(D)$.
	Suppose that $h \colon S^1 \to S^1$ is an
	orientation preserving homeomorphism,
	and let $t_0, s_0, k_0, k_1$ be as in Lemme~\ref{orientation_piecewise}.
	Then the formula \[
		\alpha(a)(t) = \begin{cases}
			a(k_1^{-1}(t)) \qquad & t \in [0, s_0] \\
			\gamma(a(k_0^{-1}(t))) \qquad & t \in [s_0, 1]
		\end{cases}
	\]
	for $a \in M_{\gamma}$ defines a $*$-automorphism
	of $M_{\gamma}$ which lies over $h$.
\end{proposition}

\begin{proof}
	We first check the continuty of $\alpha(a)$ at $s_0$. We have \[
		a(k^{-1}_1(s_0)) = a(1) = \gamma(a(0)) = \gamma(a(k_0^{-1}(s_0))) .
	\]
	So $\alpha(a) \in C([0,1], D)$. For the second, we check \[
		\alpha(a)(1) = \gamma[a(k_0^{-1}(1))] = \gamma[a(t_0)]
		= \gamma[a(k_1^{-1}(0))] = \gamma(\alpha(a)(0)).
	\]
	Thus $\alpha$ is a well-defined map on $M_{\gamma}$.
	The rest is trivial.
\end{proof}

\begin{proposition}\label{prop_larsubalg_of_mapping_torus}
	Let $D$ be a simple unital C*-algebra with nonempty tracial state space.
	Let $\gamma \in \Aut(D)$ and let $M_{\gamma}$ be as
	in Lemma~\ref{mapping_torus_locally_trivial}.
	Let $h \colon S^1 \to S^1$ be an irrational rotation and
	let $\alpha \in \Aut(A)$ be as in Lemma~\ref{action_on_ZDgamma}.
	Suppose that $\{\gamma^n \colon n \in \mathbb{Z}\}$ is pseudoperiodic.
	Fix a point $y \in S^1$.
	Then $C^*(\mathbb{Z}, M_{\gamma}, \alpha)_y$ is a large subalgebra
	of crossed product type	in $C^*(\mathbb{Z}, M_{\gamma}, \alpha)$.
\end{proposition}

\begin{proof}
	Using the atlas $\mathcal{A}$ given
	in Lemma~\ref{lem_pullback_mapping_torus},
	the corresponding set $\{\alpha(\mathcal{A})_{m,x}^{i,j}\}$ is 
	a subset of $\{\gamma^n \colon n \in \mathbb{Z}\}$.
	Then apply Theorem~\ref{thm_for_application}.
\end{proof}

Here are some concrete examples.

\begin{example}\label{example_1}
	Take $D$ to be the AF algebra $\varinjlim_n D_n$,
	with $D_n = M_{3^n} \oplus M_{3^n}$ for $n \in \mathbb{Z}_{\geq 0}$,
	and with the $*$-homomorphism $D_{n} \to D_{n+1}$ being	\[
		a \oplus b \mapsto \mathrm{diag}(a,a,b) \oplus \mathrm{diag}(b,b,a)
	\]
	for $a, b \in M_{3^n}$.
	Then $D$ is a simple unital AF algebras.
	Define $\gamma_n : D_n \to D_n$ by $\gamma_n(a,b) = (b,a)$
	for $a, b \in M_{3^n}$.
	Then the direct limit $*$-automorphism $\gamma = \varinjlim_n \gamma_n$
	exists by \cite[Thm.~1.10.16]{Lin2001Intro}.
	Let $A = \begin{pmatrix}
		2 & 1 \\ 1 & 2
	\end{pmatrix}$.
	Then $K_0(D) = \bigcup_{n=0}^{\infty} A^{-n}
	(\mathbb{Z} \oplus \mathbb{Z})$.
	A computation shows that $K_0(\gamma)$ is nontrivial on $K_0(D)$,
	so $M_{\gamma}$ is a nontrivial $C(S^1)$-algebra
	by Lemma~\ref{lem_homotopy_and_trivial_bundle}.
	Then $C^*(\mathbb{Z}, M_{\gamma}, \alpha)$ is simple and
	$C^*(\mathbb{Z}, M_{\gamma}, \alpha)_y$
	is a large subalgebra of crossed product type
	in $C^*(\mathbb{Z}, M_{\gamma}, \alpha)$ by our results.
\end{example}

\begin{example}\label{example_2}
	Take $D$ to be the reduced group C*-algebra $C^*_r(F_2)$
	of the free group on two generators.
	Then $D$ is a simple unital C*-algebras with unique tracial state
	(see \cite{Power1975simplicity}).
	Let $\gamma \in \Aut(D)$ exchange the two generators of $F_2$.
	Then $\gamma$ is not homotopic to $\mathrm{id}_D$
	because $\gamma$ is nontrivial on $K_1(D)$.
	Note that $\{\gamma^{n} \colon n \in \mathbb{Z}\}$ is pseudoperiodic
	since $\gamma^2 = \mathrm{id}$.
	Then $C^*(\mathbb{Z}, M_{\gamma}, \alpha)$ is simple and
	$C^*(\mathbb{Z}, M_{\gamma}, \alpha)_y$
	is a large subalgebra of crossed product type
	in $C^*(\mathbb{Z}, M_{\gamma}, \alpha)$.
\end{example}

\begin{proposition}\label{prop_Banach_space_decomposition}
	Let $X$ be an infinite compact Hausdorff space,
	and let $A$ be a locally trivial unital $C(X)$-algebra.
	Suppose that $h \colon X \to X$ is a minimal homeomorphism and
	$\alpha \in \Aut(A)$ lies over $h$.
	Let $Y \subset X$ be a closed subset with nonempty interior.
	Set \[
	Z_n = \begin{cases}
		\bigcup_{i=0}^{n-1} h^{i}(Y) \quad & n > 0 \\
		\varnothing  \quad & n = 0 \\
		\bigcup_{i=1}^{n} h^{-i}(Y) \quad & n < 0 .
	\end{cases}
	\]
	Then there is $N \in \mathbb{Z}_{>0}$ such that
	$C^*(\mathbb{Z}, A, \alpha)_{Y}$ has the Banach space
	direct sum decomposition \[
		C^* (\mathbb{Z}, A, \alpha)_{Y}
		= \bigoplus_{n=-N}^N C_0(X \setminus Z_n) A u^n \,.
	\]
\end{proposition}

\begin{proof}
	The proof is similar to the proof of
	Corollary 11.3.7 of \cite{GKPT2018Crossed}.
	In fact, this is the second part of Corollary~\ref{corollary_of_piece_1}
	with $J = C_0(X \setminus Y)A$.
\end{proof}

\begin{remark}\label{rek_direct_limit}
	Let $y \in X$ be arbitary.
	Suppose that $Y_1 \supset Y_2 \supset \cdots$ are closed subsets of $X$
	such that $\bigcap_{n=1}^{\infty} Y_n = \{y\}$.
	Then the algebras $C^*(\mathbb{Z}, A, \alpha)_{Y_n}$,
	for $n = 1, 2, \ldots$ form an increasing sequence such that \[
		\mathrm{clos} \bigg(
		\bigcup_{n=1}^{\infty}
			C^* \big(\mathbb{Z}, A, \alpha \big)_{Y_n}
		\bigg)
		= C^*(\mathbb{Z}, A, \alpha)_y .
	\]
	That is, $C^*(\mathbb{Z}, A, \alpha)_y
	= \varinjlim_{n} C^* \big(\mathbb{Z}, A, \alpha \big)_{Y_n}$.
\end{remark}

\begin{definition}\label{def_modified_Rokhlin_tower}
	Let $X$ be an infinite compact Hausdorff space and
	let $h\colon X \to X$ be a minimal homeomorphism.
	Let $Y \subset X$ be closed with nonempty interior.
	Denote $r_Y(x)$ to be the \emph{first return time} of $x$ to $Y$,
	i.e., $r_Y(x) = \min (\{n \geq 1 \colon h^n(x) \in Y\})$.
	The \emph{modified Rokhlin tower} associated with $Y$ consists of
	distinct numbers $n(0) < n(1) < \cdots < n(l)$,
	the closed subsets
	$Y_0, Y_1, \ldots, Y_{\ell}$
	with $Y_k = \mathrm{clos}\{x \in Y \colon r_Y(x) = n(k)\}$,
	and the open subsets
	$Y_0^{\bullet}, Y_1^{\bullet}, \ldots, Y_{\ell}^{\bullet}$ with
	$Y_k^{\bullet} = \mathrm{int}\{x \in Y \colon r_Y(x) = n(k)\}$,
	where $n(k)$, for each $k$ is the first return time to $Y$.
	See Definition~11.3.4 and Lemma 11.3.5 of \cite{GKPT2018Crossed}.
\end{definition}

\begin{lemma}\label{lem_shrink_cover}
	Let $X$ be a compact Hausdorff space.
	Let $(U_i)_{i \in I}$ be a finite open cover of $X$.
	Then there exists a subset $J \subset I$
	and an open cover $(V_j)_{j \in J}$ of $X$ such that \[
	\overline{V_{k}} \subset U_k
	\quad \mbox{\text{and}} \quad
	V_{k} \setminus \bigcup_{j \in J \setminus \{k\}} \overline{V_j}
	\ne \varnothing
	\]
	for all $k \in J$.
\end{lemma}

\begin{proof}
	Let $J$ be a subset of $I$ such that $(U_j)_{j \in J}$
	is a minimal subcover of $(U_i)_{i \in I}$.
	Then
	$U_{k} \setminus \bigcup_{j \in J \setminus \{k\}} U_j \ne \varnothing$
	for every $k \in J$.
	Since every compact Hausdorff space is normal,
	one may choose an open cover $(V_j)_{j \in J}$ of $X$
	such that $\overline{V_j} \subset U_j$ for all $j \in J$.
	Then \[
		V_{k} \setminus \bigcup_{j \in J \setminus \{k\}} \overline{V_j}
		= X \setminus \bigcup_{j \in J \setminus \{k\}} \overline{V_j}
		\supset U_{k} \setminus \bigcup_{j \in J \setminus \{k\}} U_j
		\ne \varnothing \,.	\qedhere
	\] 
\end{proof}

Let $D$ be a simple unital $C^{\ast}$-algebra.
The class of recursive subhomogeneous algebras over $D$
has been defined in \cite[Def.~3.2, 3.3]{ABP2020structure}.
We recall some notions for convenience.

\begin{notation}\label{notation_RSH}
	Every recursive subhomogeneous algebra over $D$ can be written
	in the form: \[
	R \cong \Big[ \cdots \Big[ \big[C_0 \oplus_{\varphi_1, \rho_1} C_1 \big]
	\oplus_{\varphi_2, \rho_2} C_2 \Big]
	\oplus_{\varphi_{l}, \rho_{l}} C_{l},
	\]
	with $C_k = C(X_k, M_{n(k)}(D))$ for compact Hausdorff spaces $X_k$
	and positive integers $n(k)$, and with the restriction
	homomorphisms $\rho_k \colon C_k \to C_{k}^{(0)}$
	and unital $*$-homomorphisms $\varphi_k \colon C_{k-1} \to C_{k}^{(0)}$,
	where $C_k^{(0)} = C(X_k^{(0)}, M_{n(k)}(D))$
	for compact subsets $X_k^{(0)} \subset X_k$ (possibly empty).
	Associated with this decomposition, we define the following objects.
	\begin{enumerate}
		\item Its length $l$.
		\item The $k$-th stage algebra \[
		R^{(k)} = \Big[ \cdots \Big[ \big[C_0 \oplus_{\varphi_1, \rho_1} C_1
			\big] \oplus_{\varphi_2, \rho_2} C_2 \Big]
			\oplus_{\varphi_{l}, \rho_{k}} C_{k} \,. \]
		\item Its topological dimension $\max_{k} \dim(X_k)$.
		\item Its standard representation
		$\sigma\colon R \to \bigoplus_{k=0}^{l} C(X_k, M_{n(k)}(D))$,
		defined to be the obvious inclusion.
	\end{enumerate}
\end{notation}

Let $X$ be a compact Hausdorff space,
and let $D$ be a simple unital C*-algebra.
Let $A$ be a locally trivial $C(X)$-algebra with typical fiber $D$.
By repeated use of Lemma 2.4 of \cite{Dadarlat2009continuous},
one knows that $A$ is a recursive subhomogeneous algebra over $D$.
In particular, as showned in Lemma~\ref{lem_pullback_mapping_torus},
the mapping torus $M_{\gamma}$ is a recursive subhomogeneous algebra
over $D$.
Next we show that under appropriate conditions the orbit breaking subalgebra
$C^* (\mathbb{Z}, A, \alpha)_{Y}$ is also a recursive subhomogeneous algebra.
We follow the proof in \cite[Sect.~11.3]{GKPT2018Crossed}.
Similar results play a key role in \cite{LP2004direct},
\cite{Phillips2016Cast} and \cite{ABP2020structure}.

\begin{convention}\label{convention_for_RSH}
	Let $X$ be an infinite compact Hausdorff space,
	and let $D$ be a simple unital C*-algebra.
	Let $A$ be a locally trivial $C(X)$-algebra with typical fiber $D$.
	Let $h \colon X \to X$ be a minimal homeomorphism
	and let $\alpha \in \Aut(A)$ lie over $h$.
	Suppose that
	$\mathcal{A} = \{\phi_j \colon A(\overline{U_j}) \to C(\overline{U_j}, D)\}_{j}$
	is an atlas of $A$ such that
	$(U_j)_{j}$ is a finite open cover of $X$ and
	$U_{k} \setminus \bigcup_{j \ne k} \overline{U_{j}} \ne \varnothing$
	for all $k$.
	Let $Y \subset U_1 \setminus \bigcup_{i \ne 1} \overline{U_i}$ be a
	closed subset with nonempty interior.
	Let $Z_n$ be as in Proposition \ref{prop_Banach_space_decomposition}
	and let $\{n(k), Y_k, Y_k^{\bullet}\}_{k = 0}^{l}$ be the
	modified Rokhlin tower associated with $Y$ of $(X, h)$
	as in Definition~\ref{def_modified_Rokhlin_tower}.
	Denote by $\pi_1 \colon A \to A(\overline{U_1})$
	the quotient map and let $\bar{\phi}_1 = \phi_1 \circ \pi_1$.
	Write $\beta_{m,x}^{1,1} = \alpha(\mathcal{A})_{m,x}^{1,1}$
	for every $m \in \mathbb{Z}$ and every $x \in h^m(\overline{U_1}) \cap \overline{U_1}$.
\end{convention}

\begin{proposition}\label{prop_hom_induced_by_Rokhlin_tower}
	Adopt Convention~\ref{convention_for_RSH}.
	Define the shift unitary $s_k \in M_{n(k)} \subset C(Y_k, M_{n(k)})$ by \[
	s_k = \begin{pmatrix}
		0      & 0      & \cdots & \cdots & 0      & 0      & 1 \\
		1      & 0      & \cdots & \cdots & 0      & 0      & 0 \\
		0      & 1      & \cdots & \cdots & 0      & 0      & 0 \\
		\vdots & \vdots & \ddots &        & \vdots & \vdots & \vdots \\
		\vdots & \vdots &        & \ddots & \vdots & \vdots & \vdots \\
		0      & 0      & \cdots & \cdots & 1      & 0      & 0 \\
		0      & 0      & 0      & 0      & 0      & 1      & 0
	\end{pmatrix}
	\,.\]
	For $k = 0, 1, \ldots, l$, there is a unique linear map \[
		\gamma_k \colon
		C^*(\mathbb{Z}, A, \alpha)_{Y} \to C(Y_k, M_{n(k)}(D))
	\]
	such that for $m = 0, 1,\ldots, n(k) - 1$ and $a \in C_0(X \setminus Z_m)A$
	we have:
	\begin{enumerate}
		\item $\gamma_k (a u^m) = \mathrm{diag}
		(\bar{\phi}_1(a)|_{Y_k}, \bar{\phi}_1(\alpha^{-1}(a))|_{Y_k}, \ldots,
			\bar{\phi}_1(\alpha^{-[n(k)-1]}(a))|_{Y_k}) \cdot s_k^m$.
		\item $\gamma_k (u^{-m}a) = s_k^{-m} \cdot
			\mathrm{diag}(
			\bar{\phi}_1(a)|_{Y_k}, \bar{\phi}_1(\alpha^{-1}(a))|_{Y_k}, \cdots,
			\bar{\phi}_1(\alpha^{-[n(k)-1]}(a))|_{Y_k})$.
	\end{enumerate}
	Moreover, the map \[
		\gamma \colon C^* (\mathbb{Z}, A, \alpha)_{Y}
		\to \bigoplus_{k=0}^{l} C(Y_k, M_{n(k)}(D))
	\]
	given by $\gamma(a) = (\gamma_0(a), \gamma_1(a), \ldots, \gamma_{l}(a))$
	is a $*$-homomorphism.
\end{proposition}

\begin{proof}
	The proof is similar to the proof of Lemma 11.3.9 of \cite{GKPT2018Crossed}.
\end{proof}

\begin{lemma}
	Adopt Convention~\ref{convention_for_RSH}.
	For every $i,j \in \{1, \ldots, n(k)\}$,
	let $\delta_{i,j} = 1$ for $i = j$ and $\delta_{i,j} = 0$ for $i \ne j$.
	Define maps \[
		E_k^{(m)} \colon C(Y_k, M_{n(k)}(D)) \to C(Y_k, M_{n(k)}(D))
	\quad \text{by} \enspace \, \big[a_{i,j}\big]_{i,j = 1}^{n(k)} \mapsto
	\big[\delta_{i+m,j} a_{i,j}\big]_{i,j = 1}^{n(k)} \,.\]
	Set $G_m = \bigoplus_{k=0}^{l} E_k^{(m)} \big(C(Y_k, M_{n(k)}(D))\big)$.
	Then: \begin{enumerate}
		\item\label{prop_range_part_1}
		There is a Banach space direct sum decomposition \[
			\bigoplus_{k=0}^{l} C(Y_k, M_{n(k)}(D))
			= \bigoplus_{m = -n(l)}^{n(l)} G_m.
		\]
		\item For $k = 0, 1, \ldots, l$, $m \geq 0$,
		$a \in C_0(X \setminus Z_m)A$ and $x \in Y_k$,
		setting \[
		c_1 = \mathrm{diag} \big(
			\bar{\phi}_1(\alpha^{-m}(a))(x), \bar{\phi}_1(\alpha^{-(m+1)}(a))(x),
			\ldots, \bar{\phi}_1(\alpha^{-[n(k)-1]}(a))(x)		\big),
		\]
		the expression $\gamma_k(a u^m)(x)$ is given by the following matrix,
		in which the upper right corner is a zero matrix of size $m$: \[
		\gamma_k(a u^m)(x) = \begin{pmatrix} & 0 \\ c_1 &\end{pmatrix}.
		\]
		\item For $m \geq 0$ and $a \in C_0(X \setminus Z_m)A$, we have \begin{gather*}
			\gamma_k (au^m) \in E_{k}^{(m)} \big( C(Y_k, M_{n(k)}(D)) \big),
			\quad \gamma(au^m) \in G_m, \\
			\gamma_k(u^{-m}a) \in E_{k}^{(-m)} \big( C(Y_k, M_{n(k)}(D)) \big),
			\quad \text{and} \, \enspace \gamma(u^{-m}a) \in G_{-m}.
		\end{gather*}
		\item The $*$-homomorphism $\gamma$ is compatible with the direct sum
		decomposition of Proposition~\ref{prop_Banach_space_decomposition}
		on its domain and the direct sum decomposition of
		part \eqref{prop_range_part_1} on its codomain.
	\end{enumerate}
\end{lemma}

\begin{proof}
	The proof is similar to the proof of Corollary 11.3.16 of \cite{GKPT2018Crossed}.
\end{proof}

\begin{corollary}
	Adopt Convention~\ref{convention_for_RSH}.
	The $*$-homomorphism $\gamma$ of 
	Proposition~\ref{prop_hom_induced_by_Rokhlin_tower} is injective.
\end{corollary}

\begin{proof}
	The proof is similar to the proof of Lemma 11.3.17 of \cite{GKPT2018Crossed}.
\end{proof}

\begin{lemma}
	Adopt Convention~\ref{convention_for_RSH}.
	Let \[
		b = (b_0, b_1, \ldots, b_{l})
		\in \bigoplus_{k=0}^{l} C(Y_k, M_{n(k)}(D)) .
	\]
	Then $b \in \gamma \big( C^{\ast}(\mathbb{Z}, A, \alpha)_{Y} \big)$
	if and only if whenever \begin{itemize}
		\item $r \in \mathbb{Z}_{>0}$,
		\item $k, t_1, \cdots, t_r \in \{0, \cdots, l\}$,
		\item $n(t_1) + n(t_2) + \cdots + n(t_r) = n(k)$,
		\item $x \in (Y_k \setminus Y_k^{\bullet}) \cap Y_{t_1}
			\cap h^{-n(t_1)}(Y_{t_2}) \cap \cdots \cap
			h^{-[n(t_1) + \cdots + n(t_{r-1})]}(Y_{t_r})$,
	\end{itemize}
	then $b_k(x)$ is given by the block diagonal matrix \begin{align*}
		b_{k}(x) & = \mathrm{diag} \, \Big(
			b_{t_1}(x), \,
			\beta_{-n(t_1),x}^{1,1}\big[b_{t_2}(h^{n(t_1)}x)\big], \\
			& \qquad \qquad \ldots, \,
			\beta_{-[n(t_1) + \cdots + n(t_{r-1})], x}^{1,1} \big[
				(b_{t_r}(h^{n(t_1) + \cdots + n(t_{r-1})}x)) \big]
		\Big) \,.
	\end{align*}
\end{lemma}

\begin{proof}
	The proof is similar to the proof of Lemma 11.3.18 of \cite{GKPT2018Crossed}.
\end{proof}

\begin{theorem}\label{thm_RSH}
	Let $X$ be an infinite compact Hausdorff space,
	and let $D$ be a simple unital C*-algebra.
	Let $A$ be a locally trivial $C(X)$-algebra with typical fiber $D$.
	Let $h \colon X \to X$ be a minimal homeomorphism
	and let $\alpha \in \Aut(A)$ lie over $h$.
	Suppose that
	$\mathcal{A} = \{\phi_j \colon A(\overline{U_j}) \to C(\overline{U_j}, D)\}_{j}$
	is an atlas of $A$ such that
	$(U_j)_{j}$ is a finite open cover of $X$ and
	$U_{k} \setminus \bigcup_{j \ne k} \overline{U_{j}} \ne \varnothing$
	for all $k$.
	Let $Y \subset U_1 \setminus \bigcup_{i \ne 1} \overline{U_i}$ be a
	closed subset with nonempty interior.
	Let $Z_n$ be as in Proposition \ref{prop_Banach_space_decomposition}
	and let $\{n(k), Y_k, Y_k^{\bullet}\}_{k = 0}^{l}$ be the
	modified Rokhlin tower associated with $Y$ of $(X, h)$
	as in Definition~\ref{def_modified_Rokhlin_tower}.
	Write $\beta_{m,x}^{1,1} = \alpha(\mathcal{A})_{m,x}^{1,1}$
	for every $m \in \mathbb{Z}$ and every $x \in h^m(\overline{U_1}) \cap \overline{U_1}$.
	Then the $*$-homomorphism \[
	\gamma \colon C^* (\mathbb{Z}, A, \alpha)_{Y}
		\to \bigoplus_{k=0}^{l} C(Y_k, M_{n(k)}(D))
	\]
	of Proposition~\ref{prop_hom_induced_by_Rokhlin_tower} induces a
	$*$-isomorphism of $C^* (\mathbb{Z}, A, \alpha)_{Y}$ with the
	recursive subhomogeneous algebra over $D$ defined,
	following Notation~\ref{notation_RSH},
	as follows:
	\begin{enumerate}
	\item $l$ and $n(0), n(1), \ldots, n(l)$ are as
		in Definition~\ref{def_modified_Rokhlin_tower}.
	\item $X_k = Y_k$ for $0 \leq k \leq l$.
	\item $X_k^{(0)} = Y_k \cap \bigcup_{j=0}^{k-1} Y_j$.
	\item For $x \in X_k^{(0)}$ and
		$b = (b_0, \ldots, b_{k-1})$ in the image in
		$\bigoplus_{j=0}^{k-1} C(X_{j}, M_{n(j)}(D))$ of the
		$k-1$ stage algebra $C^* (\mathbb{Z}, A, \alpha)_{Y}^{(k-1)}$,
		whenever
		\[x \in (Y_k \setminus Y_k^{\bullet}) \cap Y_{t_1}
			\cap h^{-n(t_1)}(Y_{t_2})
			\cap \cdots \cap h^{-[n(t_1) + \cdots + n(t_{r-1})]}(Y_{t_r})\]
		with $n(t_1) + n(t_2) + \cdots + n(t_r) = n(k)$,
		then the unital $*$-homomorphism
		$\varphi_k \colon C^* (\mathbb{Z}, A, \alpha)_{Y}^{(k-1)}
		\to C(X_{k}^{(0)}, M_{n(k)}(D))$
		satisfies the equation \begin{align*}
			\varphi_k (b_0, b_1, \ldots, b_{k-1})(x)
			& = \mathrm{diag} \, \Big(
				b_{t_1}(x), \,
				\beta_{-n(t_1),x}^{1,1}\big[b_{t_2}(h^{n(t_1)}x)\big], \,
				\ldots \\
				& \hspace*{5em}
				\beta_{-[n(t_1) + \cdots + n(t_{r-1})], x}^{1,1} \big[
				(b_{t_r}(h^{n(t_1) + \cdots + n(t_{r-1})}x)) \big]
				\Big) \,.
			\end{align*}
		\item $\rho_k \colon C(X_k , M_{n(k)}(D)) 
			\to C(X_k^{(0)}, M_{n(k)}(D))$ is the restriction map.
	\end{enumerate}
	The orbit breaking subalgebra $C^*(\mathbb{Z}, A, \alpha)_{Y}$
	has the following recursive decomposition of length $l$:
	\begin{align*}
	\Big[ \cdots \Big[ \big[ &
	C(X_{0},\mathrm{M}_{n(0)}(D)) \oplus_{\varphi_1 , \rho_1}
	C(X_{1}, \mathrm{M}_{n(1)}(D)) \big] \\
	& \hspace*{4em} \oplus_{\varphi_{2}, \rho_{2}} C(X_{2}, \mathrm{M}_{n(2)}(D))
	\Big] \cdots \Big]
	\oplus_{\varphi_{l}, \rho_{l}} C(X_{l}, \mathrm{M}_{n(l)}(D)) \,,
	\end{align*}
	and the topological dimension of this decomposition is $\dim(X)$.
\end{theorem}

\begin{proof}
	The proof is similar to the proof of Theorem 11.3.19 of \cite{GKPT2018Crossed}.
\end{proof}

\begin{proposition}\label{prop_RSH_structure}
	Let $X$ be an infinite compact metric space,
	and let $D$ be a simple separable unital C*-algebra.
	Let $A$ be a locally trivial $C(X)$-algebra with typical fiber $D$.
	Suppose that $h \colon X \to X$ is a minimal homeomorphism
	and $\alpha \in \Aut(A)$ lies over $h$.
	Let $y \in X$ be arbitary.
	Then $C^*(\mathbb{Z}, A, \alpha)_{y}$ is a direct limit of separable
	recursive subhomogeneous algebras over $D$ of the form
	$C^* (\mathbb{Z}, A, \alpha)_{Y_n}$	for closed subsets
	$Y_n \subset X$ with $\mathrm{int}(Y_n) \ne \varnothing$.
\end{proposition}

\begin{proof}
	Suppose that $\{\phi_j \colon A(\overline{U_j})
	\to C(\overline{U_j}, D)\}_{j=1}^n$	is a collection of
	local trivializations of $A$ such that the sets $(U_j)_{j=1}^n$ form
	an open cover of $X$. Using Lemma~\ref{lem_shrink_cover},
	one may assume that
	$U_1 \setminus \bigcup_{j = 2}^n \overline{U_j} \ne \varnothing$.
	By minimality of $h$, there exists $m \in \mathbb{Z}$ such that
	$h^m(y) \in U_1 \setminus \bigcup_{j = 2}^n \overline{U_j}$.
	Then choose a decreasing sequence $(Y_k)_{k = 1}^{\infty}$
	of closed subsets of $U_1 \setminus \bigcup_{j = 2}^n \overline{U_j}$
	with $\mathrm{int}(Y_k) \ne \varnothing$ such that
	$\bigcap_k Y_k = \{h^m(y)\}$.
	Let $\varphi \colon C^* (\mathbb{Z}, A, \alpha)
	\to C^* (\mathbb{Z}, A, \alpha)$ be the unique $*$-homomorphism
	such that \[
	\varphi\left( \sum_{k = -n}^{n} a_k u^k \right)
	= \sum_{k = -n}^{n} \alpha^m(a_k) u^k \quad \text{for} \enspace
	n \in \mathbb{Z}_{\geq 0} , \, a_{-n}, \ldots, a_{n} \in A.
	\]
	It is standard to check that $\varphi$ is a $*$-isomorphism,
	and that \[\varphi\big(C^* (\mathbb{Z}, A, \alpha)_{y}\big)
	= C^* (\mathbb{Z}, A, \alpha)_{h^m(y)}\] by Proposition~\ref{piece_1}.
	Thus $C^* (\mathbb{Z}, A, \alpha)_{y}
	\cong C^* (\mathbb{Z}, A, \alpha)_{h^m(y)}$ is a direct limit of
	recursive subhomogeneous algebras over $D$ by Theorem~\ref{thm_RSH}
	and Remark~\ref{rek_direct_limit}.
\end{proof}

\begin{proposition}
	Let $X$ be an infinite compact metric space,
	and let $D$ be a simple separable unital C*-algebra.
	Let $\mathcal{E}$ be a separable unital strongly selfabsorbing C*-algebra.
	Let $A$ be a locally trivial $C(X)$-algebra with typical fiber $D$.
	Suppose that $h \colon X \to X$ is a minimal homeomorphism
	and $\alpha \in \Aut(A)$ lies over $h$.
	If $D$ is $\mathcal{E}$-stable,
	then for any $y \in X$,	the orbit breaking subalgebra
	$C^*(\mathbb{Z}, A, \alpha)_{y}$ is $\mathcal{E}$-stable.
\end{proposition}

\begin{proof}
	Combine Proposition~\ref{prop_RSH_structure},
	\cite[Prop.~3.7]{ABP2020structure} and
	\cite[Cor.~3.4]{TW2007strongly}.
\end{proof}

By Theorem 3.1 of \cite{Winter2011strongly},
any strongly selfabsorbing C*-algebra $\mathcal{E}$ is $\mathcal{Z}$-stable.
We immediately know that
$C^*(\mathbb{Z}, A, \alpha)_{y}$ is $\mathcal{Z}$-stable.

\begin{corollary}\label{Z_stability_of_lar_subalg}
	Let $X$ be an infinite compact metric space,
	and let $D$ be a simple separable unital C*-algebra.
	Let $\mathcal{E}$ be a separable unital strongly selfabsorbing C*-algebra.
	Let $A$ be a locally trivial continuous $C(X)$-algebra
	with typical fiber $D$,.
	Suppose that $h \colon X \to X$ is a minimal homeomorphism and
	$\alpha \in \Aut(A)$ lies over $h$.
	If $D$ is $\mathcal{E}$-stable,
	then for any $y \in X$, the orbit breaking subalgebra
	$C^* (\mathbb{Z}, A, \alpha)_{y}$ is $\mathcal{Z}$-stable.
\end{corollary}

The proof of the theorem below is similar to
the proof of Theorem 4.5 of \cite{ABP2020structure}.

\begin{theorem}\label{Z_stability_of_crossed_product}
	Let $X$ be an infinite compact metric space,
	and let $D$ be a simple separable unital $\mathcal{Z}$-stable C*-algebra
	with nonempty tracial state space.
	Let $A$ be a locally trivial $C(X)$-algebra with typical fiber $D$.
	Suppose that $h \colon X \to X$ is a minimal homeomorphism
	and $\alpha \in \Aut(A)$ lies over $h$.
	Then $C^*(\mathbb{Z}, A, \alpha)$ is tracially $\mathcal{Z}$-absorbing.
	If, in addition, $A$ is nuclear,
	then $C^{\ast}(\mathbb{Z}, A, \alpha)$ is $\mathcal{Z}$-stable.
\end{theorem}

\begin{proof}
	Since $D$ is $\mathcal{Z}$-stable,
	Theorem 3.3 of \cite{HO2013tracially} implies that $D$ has
	strict comparison of positive elements.
	So $\Aut(D)$ is pseudoperiodic by \cite[Lem.~1.12]{ABP2020structure}.
	Using Corollary~\ref{Z_stability_of_lar_subalg},
	we get an orbit breaking subalgebra $B = C^* (\mathbb{Z}, A, \alpha)_{y}$
	which is $\mathcal{Z}$-stable.
	Theorem \ref{thm_for_application} implies that $B$ is a large subalgebra
	of crossed product type in $C^* (\mathbb{Z}, A, \alpha)$.
	Since $C^* (\mathbb{Z}, A, \alpha)$ is stably finite
	by Lemma~\ref{lem_finiteness},
	Theorem 3.6 of \cite{AP2020permanence} implies that $B$ is a
	centrally large subalgebra of $C^* (\mathbb{Z}, A, \alpha)$.
	Therefore $C^* (\mathbb{Z}, A, \alpha)$
	is tracially $\mathcal{Z}$-absorbing,
	by Theorem 3.2 of \cite{ABP2018centrally}.
	If $A$ is nuclear, then $C^* (\mathbb{Z}, A, \alpha)$ is also nuclear
	by \cite[Thm.~4.2.4]{BO2008Cast}.
	So $\mathcal{Z}$-stability of $C^* (\mathbb{Z}, A, \alpha)$
	follows from Theorem 3.3 of \cite{ABP2018centrally}.
\end{proof}

\begin{remark}\label{rek_known_result}
	Note that if $X$ has finite covering dimension,
	then this result is already known by the permenance of
	$\mathcal{Z}$-stability for crossed products by actions
	with finite Rokhlin dimension with commuting towers.
	Let $\tilde{h} \colon f \mapsto f \circ h^{-1}$ be the induced action on $C(X)$.
	Then $\tilde{h}$ has finite Rokhlin dimension
	with commuting towers by \cite[Thm.~6.1]{HWZ2015rokhlin}.
	Recall that the structure map is equivariant
	and its image is in the center of $A$,
	so one can check directly that
	$\dim_{\text{Rok}}^c(A, \alpha) \leq \dim_{\text{Rok}}^c(C(X), \tilde{h})$.
	Now since $A$ is a separable recursive subhomogeneous algebra over $D$,
	$A$ is $\mathcal{Z}$-stable by Proposition~3.7 of \cite{ABP2020structure}.
	Then one gets $\mathcal{Z}$-stability of $C^* (\mathbb{Z}, A, \alpha)$
	by Theorem 5.8 of \cite{HWZ2015rokhlin}.
\end{remark}

\begin{proposition}\label{prop_larsub_sr1}
	Let $X$ be an infinite compact metric space
	with covering dimension at most $1$.
	Let $D$ be a simple separable unital C*-algebra with
	$\mathrm{tsr}(D) = 1$, $\mathrm{RR}(D) = 0$ and $K_1(D) = 0$.
	Let $A$ be a locally trivial $C(X)$-algebra with typical fiber $D$.
	Suppose that $h \colon X \to X$ is a minimal homeomorphism
	and $\alpha \in \Aut(A)$ lies over $h$.
	Then for any $y \in X$,
	the orbit breaking subalgebra $C^*(\mathbb{Z}, A, \alpha)_{y}$
	has stable rank one.
\end{proposition}

\begin{proof}
	The proof is similar to the proof of Proposition 4.13 of \cite{ABP2020structure}.
\end{proof}

The proof of the theorem below is similar to the proof of Theorem 4.14 of \cite{ABP2020structure}.

\begin{theorem}\label{thm_sr1}
Let $X$ be an infinite compact metric space
with covering dimension at most $1$.
Let $D$ be a simple separable unital C*-algebra
such that $\mathrm{T}(D) \ne\varnothing$,
$\mathrm{tsr}(D) = 1$,
$\mathrm{RR}(D) = 0$ and $K_1(D) = 0$.
Let $A$ be a locally trivial $C(X)$-algebra with typical fiber $D$.
Let $h \colon X \to X$ be a minimal homeomorphism
and let $\alpha \in \Aut(A)$ lie over $h$.
Suppose that $\alpha$ is pseudoperiodic generated
in the sense of Definition~\ref{def_ppg}.
Then $C^* (\mathbb{Z}, A, \alpha)$ has stable rank one.
\end{theorem}

\begin{proof}
	Using Proposition~\ref{prop_larsub_sr1},
	we get an orbit breaking subalgebra $B$ of the form
	$C^* (\mathbb{Z}, A, \alpha)_y$ for some $y \in X$
	which has stable rank one.
	Proposition \ref{prop_larsubalg_of_mapping_torus} implies that
	$B$ is a large subalgebra of crossed product type
	of $C^* (\mathbb{Z}, A, \alpha)$.
	Since $C^* (\mathbb{Z}, A, \alpha)$ is stably finite
	by Lemma~\ref{lem_finiteness},
	Theorem 3.6 of \cite{AP2020permanence} implies that $B$ is a
	centrally large subalgebra of $C^* (\mathbb{Z}, A, \alpha)$.
	Therefore $C^* (\mathbb{Z}, A, \alpha)$ has stable rank one
	by Theorem 5.3 of \cite{AP2020permanence}.
\end{proof}

\begin{lemma}\label{nuc_mapping_torus}
	Let $D$ be a simple unital C*-algebra. 
	Let $\gamma \in \Aut(D)$ and let $M_{\gamma}$
	be as in Lemma~\ref{mapping_torus_locally_trivial}.
	If $D$ is nuclear, then $M_{\gamma}$ is also nuclear.
\end{lemma}

\begin{proof}
	There is a short exact sequence \[
	0 \to C_0(S^1 \setminus \{1\}) \otimes D \to M_{\gamma}
	\to D \to 0 .
	\]
	The ideal and quotient are nuclear,
	and it is known that nuclearity is preserved by extensions.
\end{proof}

In Example \ref{example_1}, the crossed product
$C^*(\mathbb{Z}, M_{\gamma}, \alpha)$ is $\mathcal{Z}$-stable by
Theorem~\ref{Z_stability_of_crossed_product} and Lemma~\ref{nuc_mapping_torus}.
As pointed out in Remark~\ref{rek_known_result},
this result can be obtained by using known theorems involving Rokhlin dimension.
We also get stable rank one by Theorem~\ref{thm_sr1};
this result is also known by
\cite[Thm.~6.7]{Rordam2004stable} (see also \cite[Sect.~8]{LN2020stable})
if we admit the $\mathcal{Z}$-stability of $C^*(\mathbb{Z}, M_{\gamma}, \alpha)$.
In Example \ref{example_2},
however,
we do not know whether the corresponding crossed product
$C^* (\mathbb{Z}, M_{\gamma}, \alpha)$ is tracially $\mathcal{Z}$-absorbing,
since $C^{*}_{\mathrm{r}}(F_2)$ is neither nuclear nor $\mathcal{Z}$-stable
(see \cite[Sect.~2]{GJS2000obstructions}).

\appendix{}

\section*{Acknowledgments} The first author was supported by a grant from
the National Natural Sciences Foundation of China (No.11871375).
The second author was supported by the US National Science Foundation under Grant DMS-2055771.

\bibliographystyle{plain}
\bibliography{larsubalg_bundle}

\end{document}